\documentclass{amsart}

\input xy
\xyoption{all}

\usepackage{geometry}
\usepackage{amsmath}

\usepackage{amssymb}
\usepackage{hyperref}
\usepackage{graphicx}
\usepackage{hyperref}
\usepackage{setspace}
\usepackage{amsthm}  
\newtheorem{same}{This should never appear}[section]

\newtheorem{obs}[same]{Observation}
\newtheorem{theorem}[same]{Theorem}
\newtheorem{example}[same]{Example}
\newtheorem{lemma}[same]{Lemma}

\newtheorem{question}[same]{Question}
\newtheorem{cor}[same]{Corollary}

\newtheorem{prop}[same]{Proposition}

\theoremstyle{definition}
\newtheorem{defin}[same]{Definition}
\numberwithin{equation}{section}

\newbox\noforkbox \newdimen\forklinewidth
\setbox0\hbox{$\textstyle\smile$}
\setbox1\hbox to \wd0{\hfil\vrule width \forklinewidth depth-2pt
 height 10pt \hfil}
\setbox\noforkbox\hbox{\lower 2pt\box1\lower 2pt\box0\relax}
\def\unionstick{\mathop{\copy\noforkbox}\limits}

\def\nonfork_#1{\unionstick_{\textstyle #1}}

\setbox0\hbox{$\textstyle\smile$}
\setbox1\hbox to \wd0{\hfil{\sl /\/}\hfil}
\setbox2\hbox to \wd0{\hfil\vrule height 10pt depth -2pt width
              \forklinewidth\hfil}
\newbox\doesforkbox
\setbox\doesforkbox\hbox{\lower 2pt\box1 \lower 2pt\box2\lower2pt\box0\relax}
\def\nunionstick{\mathop{\copy\doesforkbox}\limits}

\def\fork_#1{\nunionstick_{\textstyle #1}}

\newcommand{\M}{\mathcal{M}}

\newcommand{\dom}{\textrm{dom }}

\newcommand{\cf}{\text{cf }}
\newcommand{\rest}{\upharpoonright}

\newcommand{\rank}{\text{rank }}

\newcommand{\crit}{\text{crit }}

\newcommand{\seq}[1]{\langle #1 \rangle}

\newcommand{\one}{\mathop{1\hskip-2.5pt {\rm l}}}
\newcommand{\forces}{\Vdash}
\newcommand{\p}{\mathbb P}
\newcommand{\Add}{\rm Add}

\newcommand{\bA}{\mathbb{A}}

\newcommand{\bL}{\mathbb{L}}

\newcommand{\cF}{\mathcal{F}}
\newcommand{\cL}{\mathcal{L}}
\newcommand{\cM}{\mathcal{M}}
\newcommand{\cN}{\mathcal{N}}
\newcommand{\cP}{\mathcal{P}}

\newcommand{\fC}{\mathfrak{C}}
\newcommand{\fF}{\mathfrak{F}}

\newcommand{\fR}{\mathfrak{R}}

\newcommand{\im}{\text{im }}

\newcommand{\comment}[1]{}

\newcommand{\Fraisse}{Fra\"\i ss\'e}

\newcommand{\G}{\Game}

\newcommand{\Str}{\text{Str }}

\newcommand\Vopenka{Vop\v{e}nka}
\newcommand{\image}{\mathbin{\hbox{\tt\char'42}}}

\title{Model Theoretic Characterizations of Large Cardinals Revisited}

\author{Will Boney}
\address[W. Boney]{Mathematics Department, Texas State University, San Marcos, TX, USA}
\email{wb1011@txstate.edu}
\author{Stamatis Dimopoulos}
\address[S. Dimopoulos]{Department of Mathematics and Statistics, University of Cyprus, Nicosia, Cyprus}
\email{stamatiosdimopoulos@gmail.com}
\author{Victoria Gitman}
\address[V. Gitman]{Mathematics Department, CUNY Graduate Center, NY, USA}
\email{vgitman@gmail.com}
\author{Menachem Magidor}
\address[M. Magidor]{Institute of Mathematics, Hebrew university of Jerusalem, Jerusalem 91904, Israel}
\email{mensara@savion.huji.ac.il}

\date{\today\\
2020 MSC Classifications: 03E55, 03C95, 03C75, 03C55}

\begin{document}

\begin{abstract}
In \cite{b-mtlc}, model theoretic characterizations of several established large cardinal notions were given. We continue this work, by establishing such characterizations for Woodin cardinals (and variants), various virtual large cardinals, and subtle cardinals.
\end{abstract}

\maketitle
\section{Introduction}

The compactness of strong logics and set theory have been intertwined since Tarski \cite{t-compactness} defined (weakly and strongly) compact cardinals in terms of the properties of the infinitary logic $\bL_{\kappa, \omega}$. This established a strong connection between abstract model theory and the theory of large cardinals, which has also become apparent by the recent breakthroughs in the theory of Abstract Elementary Classes--a purely model theoretic framework--where certain important results depend on the existence of large cardinal axioms (e.g., \cite{b-tamelc, bu-lctame, sv-multidim}).

The interaction between the two fields is also strengthened by the first author's article \cite{b-mtlc}, which establishes new characterizations of established large cardinal notions, expressed in model theoretic terms as compactness properties. This paper is a sequel to \cite{b-mtlc} and characterizes more large cardinals this way, namely Woodin, various virtual large cardinals and subtle cardinals. Notably, this has led us to generalise or define new concepts in abstract model theory, that may be useful outside the scope of the current exposition.

One of the main philosophical open questions about the large cardinal hierarchy is to explain the fact that it appears to be linear. Hence, apart from the intrinsic interest, we believe that the framework of compactness principles that we invoke offers a new insight into this problem.

The structure of the paper is as follows. In Section \ref{sec:prelim} we fix our notation and terminology and recall definitions and known results from abstract model theory and large cardinals. In Section \ref{sec:woodin} we give a model-theoretic characterisation of Woodin cardinals by introducing a notion of Henkin models for arbitrary abstract logics. In Section \ref{sec:virtual} we characterise various virtual large cardinals by introducing the notion of a pseudo-model for a theory. Finally, in Section \ref{sec:vopenka-weak} we characterise (a class version of) subtle cardinals as a natural weakening of Vop\v{e}nka's principle by showing (in an appropriate second-order set theory) that if ${\rm Ord}$ is subtle, then every abstract logic has a stationary class of weak compactness cardinals.

The conversations leading to this paper began at the conference ``Accessible categories and their connections'' organized by Andrew Brooke-Taylor, and we would like to thank Brooke-Taylor for organizing a fascinating meeting.

\section{Logics and large cardinals}\label{sec:prelim}

\subsection{Abstract logics} We begin by fixing a notion of abstract logic.  The following is a sublist of standard properties, e.g., that appear in \cite[Definition 2.5.1]{changkeisler}, and we have omitted the clauses that do not factor into our analysis (e.g., the closure and quantifier properties).  From the definition of a language, it might appear that we have restricted ourselves to single-sorted, first-order structures.  However, many-sorts, higher-order relations, etc. can be coded into this framework.

\begin{defin}\
\begin{enumerate}
	\item A \emph{language} $\tau$ is a collection of function and relation symbols that come with a finite number as an arity, as well as constant symbols.  Formally, this means that $\tau$ is an ordered quadruple $(\fF, \fR, \fC, n)$ where $\fF$, $\fR$, and $\fC$ are disjoint sets and $n:\fF\cup\fR \to \omega$ is the arity function.
	\item Given a language $\tau$, $\Str \tau$ is the collection of all $\tau$-structures $M$, which consist of $\seq{|M|, F^M, R^M,c^M}_{F\in \fF, R \in \fR, c\in\fC}$, where $|M|$ is a set (called the \emph{universe} or \emph{underlying set} of $M$); $F^M:|M|^{n(F)} \to |M|$, $R^M \subseteq |M|^{n(F)}$, and $c^M\in |M|$.  We often do not notationally distinguish between $M$ and $|M|$.
	\item A \emph{morphism} $f$ between two languages $\tau = (\fF, \fR,\fC, n)$ and $\rho = (\fF', \fR',\fC', n')$ is an injective function $f:\fF\cup\fR\cup\fC \to \fF'\cup\fR'\cup\fC'$ that preserves the partition and maintains the arity.  A \emph{renaming} is a bijective morphism.  Note that a renaming\footnote{Technically, a renaming is not a map from $\tau$ to $\rho$, but rather a map between the unions of the components of each, but we will employ this abuse of notation here for clarity of presentation.} $f:\tau\to \rho$ induces a bijection $f^*$ from $\Str \tau$ to $\Str \rho$ that fixes the underlying sets.
	\item A logic is a pair of classes $(\cL, \vDash_\cL)$ satisfying the following conditions.
	\begin{enumerate}
		\item $\cL$ is a (class) map from languages and we call $\cL(\tau)$ the set of $\tau$-sentences.
		\item $\vDash_\cL \subseteq\bigcup_{\text{languages }\tau} \Str \tau \times \cL(\tau)$ is the satisfaction relation.
		\item (monotonicity) If $\tau \subseteq \rho$, then $\cL(\tau) \subseteq \cL(\rho)$.
		\item (expansion) If $\phi \in \cL(\tau)$, $\rho\supseteq \tau$, and $M$ is a $\rho$-structure, then $M \vDash_\cL \phi$ if and only if the reduct of $M$ to $\tau$, $M \rest \tau \vDash_\cL \phi$.
		\item (isomorphism) If $M \cong N$, then $M \vDash_\cL \phi$ if and only if $N \vDash_\cL \phi$.
		\item (renaming) Every renaming $f:\tau\to\rho$ induces a unique bijection \hbox{$f_*:\cL(\tau) \to \cL(\rho)$} such that, for any $\tau$-structure $M$ and $\phi \in \cL(\tau)$, we have
		$$M\vDash_\cL \phi \text{ if and only if } f^*(M) \vDash_\cL f_*(\phi).$$
	\end{enumerate}
	We often refer to a logic as just $\cL$ and drop the subscript from satisfaction--simply writing $\vDash$--when the context makes it clear.
	
	\item The \emph{occurrence number} of a logic $\cL$--written $o(\cL)$--is the minimal cardinal $\kappa$ such that, for every $\phi \in \cL(\tau)$, there is $\tau_0 \in \cP_\kappa \tau$ such that $\phi \in \cL(\tau_0)$.\footnote{The notation $\cP_\kappa A$ denotes the collection of all subsets of $A$ of size less than $\kappa$.}
	
\end{enumerate}
\end{defin}
Note that abstract logics are defined only for sentences, there is no incorporation of free variables, although these can be tacitly handled by adding and properly interpreting constants. So using this work around, we can, in fact, assume that free variables are available. Also, note that we are requiring abstract logics to have an occurrence number. By definition, all our languages $\tau$ are set-sized and there can be only set-many sentences $\cL(\tau)$ in a fixed language.

It will be useful for later to note that the unique bijections $f_*$ associated to renamings $f$ respect both composition and restriction. More precisely if $f:\tau\to\sigma$ and $g:\sigma\to\rho$ are renamings, then clearly $g\circ f$ is a renaming and, by uniqueness, it must be the case that $g_*\circ f_*=(g\circ f)_*$. Also, if $\sigma\subseteq \tau$ are languages and $f:\tau\to\rho$ is a renaming, then clearly $f\restriction\sigma:\sigma\to f\image\sigma$ is a renaming, and again by uniqueness, $(f\restriction\sigma)_*=f_*\restriction\sigma$.

The intuition behind most of the properties of an abstract logic is clear. The occurrence number captures our intuition that there should be a bound on the number of elements of a language that a single assertion can reference. For instance, first-order logic $\bL_{\omega,\omega}$ has occurrence number $\omega$ because no single assertion can mention more than finitely much of the language and infinitary logics $\bL_{\kappa,\omega}$ (see below for definition) have occurrence number $\kappa$.

We will often consider unions of logics. If $\cL_0$ and $\cL_1$ are logics, then $\cL_0 \cup \cL_1$ is the natural union of them, with sentences identified if they are satisfied by the same models.

An $\cL$-theory is ${<}\kappa$-satisfiable when every ${<}\kappa$-sized subset of it has a model. A cardinal $\kappa$ is a \emph{strong compactness cardinal} of a logic $\cL$ if every ${<}\kappa$-satisfiable $\cL$-theory is satisfiable. A cardinal $\kappa$ is a \emph{weak compactness cardinal} of a logic $\cL$ if every ${<}\kappa$-satisfiable $\cL$-theory of size $\kappa$ is satisfiable. For example, $\omega$ is the strong compactness cardinal of first-order logic, a weakly compact cardinal $\kappa$ is a weak compactness cardinal of the infinitary logic $\bL_{\kappa,\kappa}$, and a strongly compact cardinal $\kappa$ is a strong compactness cardinal of $\bL_{\kappa,\kappa}$. Makowsky showed that every logic has a strong compactness cardinal if and only if Vop\v{e}nka's principle holds \cite[Theorem 2]{m-vopcomp} (see Section~\ref{sec:virtual} for more details).

Another very useful, but less commonly known, compactness property is chain compactness. We will say that a cardinal $\kappa$ is a \emph{chain compactness cardinal} for a logic $\cL$ if every theory $\cL$-theory $T$, which can be written as an increasing union $T=\bigcup_{\eta<\kappa}T_\eta$ of satisfiable theories, is satisfiable. Note, in particular, if $\kappa$ is a chain compactness cardinal for $\cL$, then it is a weak compactness cardinal for $\cL$ because any ${<}\kappa$-satisfiable theory of size $\kappa$ can be written as an increasing chain of length $\kappa$ of satisfiable theories.

Next, we will give an overview of some specific logics that come up in the article and their key properties.  To distinguish abstract logics from a specific logic, we use $\cL$ to denote abstract logics and $\bL$ (with some decoration) to denote specific logics.

Given cardinals $\mu\leq\kappa$, the logic $\bL_{\kappa,\mu}$ extends first-order logic by closing the rules of formula formation under conjunctions (and disjunctions) of ${<}\kappa$-many formulas that are jointly in ${<}\mu$-many free variables, and under existential (and universal) quantification of ${<}\mu$-many variables. $\bL_{\omega,\omega}$ is just first-order logic and is typically denoted $\bL$, and if $\kappa$ or $\mu$ are uncountable, we refer to $\bL_{\kappa,\mu}$ as an infinitary logic.  As we already alluded to above, compactness and other properties of infinitary logics are connected to the existence of large cardinals.

Second-order logic $\bL^2$ extends first-order logic by allowing quantification over all relations on the universe (from the overarching universe $V$ of sets). In structures where coding is available, such as arithmetic or set theory, this reduces to quantification over all subsets of the universe. We will often make use of the fact that in $\bL^2(\{\in\})$, there is an  assertion, known as Magidor's $\Phi$, encoding the well-foundedness of $\in$ and that the model is isomorphic to some $V_\beta$; Magidor's $\Phi$ is used in proofs in \cite{m-roleof} and is explicitly discussed at \cite[Fact 2.1]{b-mtlc}. We can also extend the second-order logic $\bL^2$ by allowing infintary conjunctions and quantification, resulting in logics $\bL_{\kappa,\mu}^2$.

The logic $\bL(Q^{WF})$ is first-order logic augmented by the quantifier $Q^{WF}$ that takes in two variables so that $Q^{WF}xy\varphi(x,y)$ is true if $\varphi(x,y)$ defines a well-founded relation: there is no sequence $\langle x_n\mid n<\omega\rangle$ such that $\varphi(x_{n+1},x_n)$ holds for all $n<\omega$. Note that $\bL(Q^{WF})\subseteq \bL_{\omega_1,\omega_1}\cap \bL^2$ since the quantifier $Q^{WF}$ is expressible in each of these logics.

A particularly powerful logic extending second-order logic is \emph{sort logic} $\bL^{s}$, which was introduced by V\"a\"an\"anen (see \cite{vaananen:sortLogic} for a precise definition and properties). It has $\omega$-many sorts, each with its own universe of objects, and allows predicate quantifiers over all the relations on the sorts (this is essentially a generalization of second-order logic to $\omega$-many sorts). The crucial feature of sort logic are \emph{sort quantifiers} $\tilde\exists$ and $\tilde\forall$ which range over all sets in $V$ (not just relations on a sort) searching for additional universes satisfying some desired relations. A sort quantifier $\tilde\exists X$ answers the question about whether the model can be expanded to include universes with finitely many new sorts satisfying the relation $X$. Since sort logic extends second-order logic, we can, in particular, use Magidor's sentence $\Phi$ to pick out the structures $(V_\alpha,\in)$, and now using the power of sort quantifiers we will be able to express that $V_\alpha$ is $\Sigma_n$-elementary in $V$.
\begin{prop}\label{prop:sort}
In sort logic $\bL^s$, we can express that a universe of a given sort, with a binary relation on it, is (isomorphic to) $(V_\alpha,\in)$ with $V_\alpha\prec_{\Sigma_n} V$. This assertion has complexity $\Sigma_n$ for the sort quantifiers.
\end{prop}
\begin{proof}
We use Magidor's $\Phi$ to express that the universe is some $V_\alpha$. Next, we argue, by induction on complexity, that for every formula $\phi(x)$ in the language of set theory, there is a corresponding sentence $\phi^*$ of sort logic such that $V_\alpha$ reflects $V$ with respect to $\phi(x)$ if and only if $V_\alpha$ satisfies $\phi^*$ in sort logic. For the base case, observe that $V_\alpha$ already reflects $V$ with respect to all $\Delta_0$-assertions. For a $\Sigma_1$-assertion $\phi(x):=\exists y\,\psi(y,x)$ with $\psi(y,x)$ being $\Delta_0$, we let $\phi^*$ informally be the formula $\forall a\in V_\alpha\,V_\alpha\vDash\phi(a)\leftrightarrow\exists V_\delta\, (\exists y\in V_\delta\, V_\delta\vDash\psi(y,a))$. We say here ``informally" because in actuality we would have to say that there exists a predicate satisfying Magidor's $\Phi$ and an embedding of $V_\alpha$ into the universe of this predicate so that we have $\psi(y,a')$ holds, where $a'$ is the image of $a$ under the isomorphism, etc. For $\Pi_1$-formulas, we replace the $\exists V_\delta\,\exists y$ quantifiers by $\forall V_\delta\,\forall y$ and for formulas of complexity $n+1$, we use the translation for formulas of complexity $n$ to quantify only over $V_\delta$'s that are sufficiently elementary in $V$. Note that for $\Sigma_1$-formulas $\phi(x)$, the sentence $\phi^*$ has complexity $\Sigma_1$ for sort quantifiers because the only sort quantifier is ``$\exists V_\delta$", and correspondingly, for $\Sigma_n$-formulas $\phi(x)$, the complexity of $\phi^*$ will be $\Sigma_n$.
\end{proof}
Because defining a satisfaction relation for $\bL^{s}$ runs into definability of truth issues, we limit our analysis to logics $\bL^{s,\Sigma_n}$ where we are only allowed to use $\Sigma_n$-formulas with sort quantifiers.


As a curiosity, observe that, for example, the logic $\bL_{{\rm Ord},{\omega}}$ is not an abstract logic under our criteria because, in particular, there is a proper class of sentences for a given language. This logic has several other undesirable properties as well: it does not have an occurrence number and it can never have a weak compactness cardinal. We will call such logics \emph{quasi-logics}.
\subsection{Large cardinals}\label{sub:largecardinals}

We collect here several of the large cardinal notions that we use.  Occasionally, we defer a definition to later if it is tailored to a specific situation.

First, we have several variants of Woodin cardinals.  The notion of externally definable Woodin cardinals is new.  The definition of Woodin for strong compactness (due to Dimopoulos) is phrased in the equivalent form of \cite[Proposition 3.3]{d-wfsc}.

\begin{defin}\
\begin{enumerate}
\item A cardinal $\kappa$ is \emph{$\alpha$-strong for a set $A$} if there is an elementary embedding $j:V\to\cM$ with
\begin{enumerate}
\item $\crit j=\kappa$,
\item $j(\kappa)>\alpha$,
\item $V_\alpha\subseteq \cM$,
\item $j(A)\cap V_\alpha=A\cap V_\alpha$.
\end{enumerate}
A cardinal $\kappa$ is \emph{${<}\delta$-strong for a set $A$} if it is $\alpha$-strong for $A$ for every $\kappa<\alpha<\delta$.
\item A cardinal $\kappa$ is \emph{$\alpha$-strongly compact for a set of ordinals $A$} if there is an elementary embedding $j:V\to \cM$ with
\begin{enumerate}
\item $\crit j=\kappa$,
\item $j(\kappa)>\alpha$,
\item $j(A)\cap \alpha=A\cap \alpha$,
\end{enumerate}
and there is $s \in \cM$ with $|s|^\cM < j(\kappa)$ and $j\image\alpha \subseteq s$. A cardinal $\kappa$ is \emph{${<}\delta$-strongly compact for a set of ordinals $A$} if it is $\alpha$-strongly compact for $A$ for every $\kappa<\alpha<\delta$.

	\item A cardinal $\delta$ is \emph{Woodin} if for all $A \subseteq V_\delta$, there is $\kappa<\delta$ which is ${<}\delta$-strong for $A$.\footnote{It would be equivalent to replace the requirement ``for all $A\subseteq V_\delta$" with the requirement ``for all sets $A$". The formulation we use is more standard because it emphasizes that no information outside $V_\delta$ is required.}
	\item A cardinal $\delta$ is \emph{externally definable Woodin} if it satisfies the definition of a Woodin cardinal when restricting the $A$'s to be externally definable sets.  Explicitly, this means that for every formula $\phi(x,a)$ with $a\in V_\delta$, there is $\kappa < \delta$ such that $a\in V_\kappa$ and for all $\kappa <\alpha < \delta$, there is an elementary embedding $j:V \to \cM$ with
\begin{enumerate}
    \item $\crit j = \kappa$,
    \item $j(\kappa)> \alpha$,
    \item $V_\alpha \subseteq \cM$, and
    \item $\phi(\cM,a) \cap V_\alpha =  \phi(V,a) \cap V_\alpha$.
\end{enumerate}
Here, a set $A$ has been replaced with the definable class $\phi(V,a)$.

\item A cardinal $\delta$ is \emph{Woodin for strong compactness} if for every $A\subseteq \delta$ there is $\kappa<\delta$ which is ${<}\delta$-strongly compact for $A$.

\end{enumerate}
\end{defin}

It should be clear that Woodin cardinals are externally definable Woodin. It is also not difficult to see that if $\delta$ is a Woodin cardinal, then $V_\delta$ is a model of proper class many externally definable Woodin cardinals. Woodin cardinals are Mahlo, but not necessarily weakly compact (because being Woodin is a $\Pi^1_1$-property). In \cite{DimopolousGitman:Woodin}, we show that consistently externally definable Woodin cardinals can be singular of cofinality $\omega$ or inaccessible, but not Mahlo. We also show that a Mahlo externally definable Woodin cardinal need not be Woodin.

It should be noted that in the definition of $\alpha$-strongly compact cardinals for a set $A$, we need $A$ to be a set of ordinals (instead of an arbitrary set). This is because strongly compact embeddings do not necessarily possess strongness degrees, so properties of the form $A\cap V_\alpha=j(A)\cap V_\alpha$, may not make sense if $V_\alpha$ is not the same in the target model of the embedding $j$. Nevertheless, an equivalent definition of $\delta$ being Woodin for strong compactness is that for each $A\subseteq V_\delta$, there is $\kappa<\delta$ which is both ${<}\delta$-strong for $A$ and ${<}\delta$-strongly compact, and for each $\alpha<\delta$ these two properties can be witneseed by the same embedding. For the full proof of this characterization, see Theorem 3.3. in \cite{d-wfsc}.

The next class of large cardinals we will consider are the recently introduced virtual large cardinals. Given a set-theoretic property $P$ characterized by the existence of elementary embeddings between (set) first-order structures, we say that $P$ holds \emph{virtually} if embeddings characterizing $P$ between structures from $V$ exist in set-forcing extensions of $V$.

Since most large cardinals can be characterized by the existence of elementary embeddings between (set) models of set-theory (in the case of class embeddings $j:V\to \cM$ we chop off the universe at an appropriate rank initial segment), they are natural candidates for virtualization. The study of virtual large cardinals was initiated by Schindler when he introduced the notion of a remarkable cardinal and showed that it is equiconsistent with the assertion that the theory of $L(\mathbb R)$ cannot be changed by proper forcing \cite{schindler:remarkable1}. He later observed that remarkable cardinals have an equivalent characterization as virtually supercompact cardinals. Other virtual large cardinals were subsequently studied in \cite{gs-virtual}. Unlike their philosophical cousins, the generic large cardinals, virtual large cardinals are actual large cardinals (they are ineffable and more), but they sit much lower in the hierarchy than their original counterparts. They are consistent with $L$ and are in the neighborhood of an $\omega$-Erd\H{o}s cardinal. In the definitions of virtual large cardinals given below, we will abbreviate the statement that an elementary embedding exists in some forcing extension by saying that there is a ``virtual elementary embedding".

The following absoluteness lemma for the existence of embeddings on a countable structure has crucial implications for the theory of virtual embeddings.
\begin{lemma}\label{lem:absolutenessLemma}
Suppose that $M$ is a countable first-order structure and $j:M\to N$ is an elementary embedding. If $W$ is a transitive (set or class) model of (some sufficiently large fragment of) ${\rm ZFC}$ such that $M$ is countable in $W$ and $N\in W$, then for any finite subset of $M$, $W$ has some elementary embedding $j^*:M\to N$, which agrees with $j$ on that subset. Moreover, if both $M$ and $N$ are transitive $\in$-structures and $j$ has a critical point, we can additionally assume that $\crit(j^*)=\crit(j)$.
\end{lemma}

\begin{cor}
Suppose that $M$ and $N$ are first-order structures in a common language such that there is a virtual elementary embedding $j:M\to N$. Then for any finite subset of $M$, every collapse extension ${\rm Coll}(\omega,M)$ has an elementary embedding $j^*:M\to N$, which agrees with $j$ on that subset. Moreover, if both $M$ and $N$ are transitive $\in$-structures and $j$ has a critical point, we can additionally assume that $\crit(j^*)=\crit(j)$.
\end{cor}
\noindent The proofs of the lemma and the corollary can be found in \cite{gs-virtual}.


\begin{defin}\
\begin{enumerate}
	\item A cardinal $\kappa$ is \emph{virtually measurable} if for every $\alpha > \kappa$, there is a transitive $\cM$ with $\cM^{{<}\kappa}\subseteq \cM$ such that there is a virtual elemetary embedding $j:V_\alpha \to \cM$ with $\crit j = \kappa$.
	\item A cardinal $\kappa$ is \emph{virtually supercompact} (remarkable) if for every $\lambda > \kappa$, there is $\alpha>\lambda$ and a transitive $\cM$ with $\cM^\lambda\subseteq \cM$ such that there is a virtual elementary embedding $j:V_\alpha \to \cM$ with $\crit j=\kappa$ and $j(\kappa)>\lambda$.
\item A cardinal $\kappa$ is \emph{virtually extendible} if for every $\alpha > \kappa$, there is a virtual elementary embedding $j:V_\alpha \to V_\beta$ with $\crit j = \kappa$ and $j(\kappa) > \alpha$. A cardinal $\kappa$ is \emph{weakly virtually extendible} if we omit the assumption that $j(\kappa)>\alpha$.\\[5pt]
Let $C^{(n)}$ be the class club of cardinals $\alpha$ such that $V_\alpha\prec_{\Sigma_n} V$.
\item A cardinal $\kappa$ is \emph{virtually $C^{(n)}$-extendible} if\footnote{Bagaria's original definition of $C^{(n)}$-extendibility required only that $V_{j(\kappa)}\prec_{\Sigma_n}V$, but the third author and Hamkins \cite{gh-genvop} and Tsarpounis \cite[Corollary 3.5]{t-cn-ext} (independently) showed that the two definitions are equivalent.} for every $\alpha>\kappa$ in $C^{(n)}$, there is $\beta\in C^{(n)}$ and a virtual elementary embedding $j:V_\alpha\to V_\beta$ with $\crit j=\kappa$ and $j(\kappa)>\alpha$. A cardinal $\kappa$ is \emph{weakly virtually $C^{(n)}$-extendible} if we omit the assumption that $j(\kappa)>\alpha$.
\end{enumerate}
\end{defin}
Although, for the most part the hierarchy of virtual large cardinals mirrors the hierarchy of their original counterparts, but much lower down, there are anomalies that appear to arise from the following two circumstances. The first is that Kunen's Inconsistency does not hold for virtual large cardinals, meaning that we can have virtual elementary embeddings $j:V_\alpha\to V_\alpha$ with $\alpha$ much larger than the supremum of the critical sequence of $j$. This accounts for the split in the definition of virtually extendible cardinals into the weak and strong forms given above. In the case of the actual extendible cardinals, we can argue using Kunen's Inconsistency that the assumption $j(\kappa)>\alpha$ is superfluous, which gives that the weak extendible and extendible cardinals are equivalent. But the equivalence fails in a surprising way in the virtual context, where it is consistent that there are weakly virtually extendible cardinals that are not virtually extendible, and moreover the consistency strength of the existence of such a notion is higher than that of the existence of a virtually extendible cardinal \cite{gh-genvop}. The second circumstance is that the more robust virtual large cardinals arise from large cardinals that have characterizations in terms of the existence of virtual embeddings between rank initial segments $V_\alpha$, such as rank-into-rank, supercompact, and extendible cardinals. Virtual versions of large cardinal notions lacking such characterizations do not fit properly into the hierarchy. For example, the following is not difficult to see.
\begin{theorem}[\cite{gs-virtual}]
A cardinal $\kappa$ is virtually supercompact if and only if for every $\lambda > \kappa$, there is $\alpha>\lambda$ and a transitive $\cM$ with $V_\lambda\subseteq \cM$ such that there is a virtual elementary embedding $j:V_\alpha \to \cM$ with $\crit j=\kappa$ and $j(\kappa)>\lambda$.
\end{theorem}
\noindent It follows that the notions of virtually strong and virtually supercompact cardinals coincide, and one should note here that strong cardinals do not have a characterization in terms of existence of embeddings between rank initial segments. This phenomena is also exhibited by the virtually measurable cardinals introduced and studied by Nielsen and Welch \cite{nw-games-ramsey}\footnote{Note that Nielsen's original definition did not have the closure assumption on the target models $\cM$.}, who showed that virtually measurable cardinals and virtually supercompact cardinals are equiconsistent (either a virtually measurable cardinal is virtually supercompact in $L$ or there is a virtually rank-into-rank cardinal in $L$).

Finally, Section \ref{sec:vopenka-weak} uses the hypothesis that ${\rm Ord}$ is subtle. Recall that a regular cardinal $\kappa$ is \emph{subtle} if for every club $C\subseteq\kappa$ and sequence $\langle A_\alpha\mid\alpha<\kappa\rangle$ with $A_\alpha\subseteq\alpha$, there are $\alpha<\beta$ in $C$ such that $A_\alpha=A_\beta\cap\alpha$.
\begin{defin}\label{def-subtle}
We say that \emph{{\rm Ord} is subtle} if for every class club $C\subseteq{\rm Ord}$ and every class sequence $\langle A_\alpha\mid\alpha\in{\rm Ord}\rangle$ with $A_\alpha\subseteq\alpha$, there are $\alpha<\beta \in C$ such that $A_\alpha=A_\beta \cap \alpha$.
\end{defin}
The definition requires us to specify precisely what a class is in a universe of set theory. There are a number of approaches to this and we will be intentionally vague about which approach we take in this article. We can be working in first-order logic in the theory ${\rm ZFC}$ and specify that classes are definable collections. We can be working in any one of the numerous second-order set theories where classes are second-order objects, such as G\"odel-Bernays set theory ${\rm GBC}$ or the much stronger Kelley-Morse set theory ${\rm KM}$. We can also assume that our universe is the $V_\kappa$ of a much larger ${\rm ZFC}$-model in which $\kappa$ is subtle and the classes in this case are the $V_{\kappa+1}$ of this model. In the last case, $V_\kappa$ together with the classes given by $V_{\kappa+1}$ satisfy Kelley-Morse (and more). Our arguments will require the class axiom \emph{global choice} which asserts that there is a class well-ordering of the universe of sets. Both ${\rm GBC}$ and Kelley-Morse include global choice, but a universe of set theory need not have a definable well-ordering of all sets, so global choice can fail for definable classes.

\section{Woodin cardinals and abstract Henkin structures}\label{sec:woodin}
In this section, we will give compactness characterizations for Woodin cardinals and their variants. To motivate our characterizations, we start by recalling the compactness characterization of strong cardinals.

Recall that a \emph{standard} model $(M,P)$ of second-order logic has the second-order part $P$ consisting of all subsets of $M$, so that the second-order quantifiers range over all subsets of the domain. A \emph{Henkin} model $(M,P)$ has the second-order part $P\subseteq \cP(M)$ that is a possibly proper sub-collection of the subsets of $M$, so that the second-order quantifiers cannot access all subsets of the domain.
\begin{theorem}[{\cite[Theorem 4.7]{b-mtlc}}]\label{thm:strong-char}
A cardinal $\kappa$ is $\lambda$-strong if and only if  every $\bL^2_{\kappa, \omega}(Q^{WF})$ theory $T$, which can be written as an increasing union $T = \cup_{\eta < \kappa} T_\eta$ of (standardly) satisfiable theories, has a Henkin model whose universe is an ordinal and whose second-order part has all subsets of rank $<\lambda$.
\end{theorem}
\noindent Requiring that the universe of the Henkin model is an ordinal ensures that the condition of having all small subsets of the universe is non-vacuously satisfied because a Henkin model with an arbitrary universe may not have any subsets of rank $<\lambda$.  Also, the $Q^{WF}$ is required to be standard in this interpretation (otherwise, it's inclusion would be superfluous).

Due to the similarity between Woodin and strong cardinals, the compactness characterization of Woodin cardinals will also use Henkin models.  However, to accommodate the ``$A \subseteq V_\delta$" parameter, we will need a much more specialized notion of a Henkin structure for abstract logics. In a Henkin model we have a nonstandard conception of evaluating the truth of sentences (and the data to carry out this evaluation) because we do not have access to all true subsets, but the collection of sentences remain the same.  We encapsulate this idea in the following definition of Henkin structures for abstract logics.

\begin{defin}\label{defn:henkin-structure}
Fix a logic $(\cL, \vDash_\cL)$ and a language $\tau$.
\begin{enumerate}
	\item A \emph{Henkin $\tau$-structure for $\cL$} is a model of set theory $\hat{M} = (\cM,{\rm E},\vDash^*, M)$, with an additional binary relation $\vDash^*$ and a distinguished element $M$, satisfying the following properties:
\begin{enumerate}
\item $\hat M\vDash{\rm ZFC}^*$ (a large finite fragment of ${\rm ZFC}$).
\item $\tau\in \cM$, $\cL(\tau)\subseteq \cM$.
\item For $\Sigma_n$-formulas $\varphi(x)$ (where $n$ is taken to be very large), we have $\varphi(\vDash_\cL)$ if and only if $\cM \vDash \varphi(\vDash^*).$

\end{enumerate}

	\item Given an $\cL(\tau)$-theory $T$, we say that a Henkin structure $\hat M=(\cM, {\rm E},\vDash^*, M)$ \emph{Henkin-models} $T$ if for every $\phi\in T$, $\cM$ has a partial $\tau\supseteq\tau_\phi$-structure on $M$ such that $M\vDash^*\phi$ and the assignment of the partial $\tau_\phi$-structures is coherent in the sense that if $\tau_\phi\subseteq \tau_\psi$, then the structure assignments extend correspondingly.
\end{enumerate}
\end{defin}
We think of $\hat M$ as a model of set theory that is correct about the sentences in $\cL(\tau)$, but which has a nonstandard satisfaction for them given by the relation $\vDash^*$. The statement ``$\tau\in \cM$" needs further clarification. Here we intend that not only is $\tau$ an element of the universe $\cM$ of $\hat M$, but also $\cM$ sees $\tau$ as a language with its functions, relations, and constants. In practice, we will be working with Henkin $\tau$-structures for which ${\rm E}$ is $\in$ and $(\cM,\in)$ is transitive, so these issues will not arise. For our purposes, it appears to be too strong to require that a transitive $\in$-structure $\hat M$ has a total $\tau$-structure on $M$ which witnesses that $M$ satisfies $T$ according to $\vDash^*$. Thus, we require only that sufficiently large pieces of $\tau$ can be interpreted over $M$ making it satisfy $\phi\in T$ and the interpretation is coherent. It would be easy to achieve a full interpretation of $\tau$ if we were willing to give up transitivity, but this appears to us to be unnatural.

Note that, given a Henkin $\tau$-structure $\hat M=(\cM, {\rm E}, \vDash^*,M)$ that Henkin models a theory $T$, we do get in $V$ a total $\tau$-structure on $M$ that is the union of the coherent interpretations from $\cM$.

It should be noted that our set-up does not generalize the usual notion of Henkin models for second-order logic, which we described above, because it cannot accommodate models with arbitrary second-order parts, but instead restricts it only to models with a ``sensible" second-order part, namely one that arises from a model of set theory that reflects fundamental properties of second-order logic. This is a natural class of strong Henkin models for second-order logic. This restricted notion of a Henkin model for second-order logic, which lies somewhere between the standard and Henkin models, appears to be interesting in its own right.

The motivation behind the clauses of Definition \ref{defn:henkin-structure} will become apparent after the proof of Theorem \ref{thm:woodin-char}, but we will try to give some explanation here. Henkin models came up in the proof of Theorem \ref{thm:strong-char} through elementary embeddings of the form $j:V\to \cM$ with $V_\lambda\subseteq \cM$ for some cardinal $\lambda$. More precisely, for a second-order theory $T$, the proof gives a second-order model of $j\image T$ in the sense of $\cM$ and such a model can interpret the second-order variables correctly only if they refer to subsets of $V_\lambda$. It is crucial that the relation $j(\vDash_2)$ is the same as $\vDash_2$ for structures in $\cM$. Also, the renaming of $T$ to $j\image T$ is not necessarily in $\cM$.

Our proof will work with abstract logics $\cL$ and we will once again use elementary embeddings $j:V\to \cM$ to obtain models of the desired theory in $\cM$. Here, we face the challenge that the satisfaction relation $\vDash_{\cL}$ may not be the same as $j(\vDash_{\cL})$ and the latter is essentially the $\vDash^*$ relation of the definition. Also, we once again find a model for $j\image T$ instead of $T$, but since $j\image T$ may not be in $\cM$, $\cM$ may not have the right $\tau$-structure for $M$ obtained via the renaming $j:\tau\to j\image \tau$, although it will have coherent partial $\tau$-structures resulting from this renaming.

In order to ensure some similarity between $\cL$ and $j(\cL)$, we add Clause (1c) which guarantees that any $\Sigma_n$-definable aspect of $\cL$ (for some very large $n$) is mirrored by $\cM$.


The following are a list of useful properties of Henkin $\tau$-structures. Often, whether or not a Henkin $\tau$-structure correctly computes satisfaction for a logic will depend on these properties.  For instance, a Henkin $\tau$-structure $\hat{M}$ will correctly verify the well-foundedness quantifier $Q^{WF}$ when $\cM$ is itself well-founded.  Note that we emphasize well-foundedness rather than transitivity to make the notion closed under isomorphism.

\begin{defin}\label{defn:henkin-properties}
Given a Henkin $\tau$-structure $\hat{M} = (\cM,{\rm E}\vDash^*, M)$ for a logic $(\cL, \vDash_\cL)$ and language $\tau$, we say:
\begin{enumerate}
	\item $\hat{M}$ is \emph{well-founded} whenever $(\cM,{\rm E})$ is well-founded.
\item $\hat {M}$ is \emph{transitive} whenever ${\rm E}$ is $\in$ and $(\cM,\in)$ is transitive.
	\item $\hat{M}$ is \emph{full up to $\lambda$} whenever $V_\lambda \subseteq \cM$ (if $\cM$ is not well-founded, this means for every $x \in V_\lambda$, there is $y \in \cM$ such that ${\rm E}$ is well-founded below $y$ and the transitive collapse of $(y, {\rm E})$ is $x$).
	\item Suppose that $\hat{M}$ is transitive and full up to $\lambda$. We say $\hat{M}$ is \emph{$\cL$-correct up to $\lambda$} if $(\vDash^*)\restriction V_\lambda\times V_\lambda= (\vDash)\restriction V_\lambda\times V_\lambda$. 
\item Suppose that $\hat{M}$ is transitive and $A\subseteq \cM$. We say that $\hat{M}$ is \emph{$n$-correct for $A$} if
\begin{itemize}
\item for every $\Sigma_n$-formula $\phi(x)$ and $a\in A$, $\cM\vDash \phi(a)$ if and only if $\phi(a)$ holds,
\item for every $\Sigma_n$-formula $\varphi(x)$, $\varphi(\vDash_{\cL})$ if and only if $\hat M\vDash\varphi(\vDash^*)$.
\end{itemize}
\end{enumerate}
\end{defin}

We will now be able to characterize Woodin cardinals in terms of chain compactness for sufficiently correct Henkin $\tau$-structures.
Before we give the proof, we are going to need a notion of closure points for abstract logics, which is captured by the following result.

\begin{prop}\label{prop:closure-points}
Let $(\cL,\vDash_{\cL})$ be a logic. There is a closed unbounded class of cardinals $\alpha$ such that
\begin{enumerate}
 \item $\cL\restriction V_\alpha:V_\alpha\to V_\alpha$,
 \item $o(\cL)<\alpha$.
\end{enumerate}
\end{prop}

\noindent The proof is a standard closure argument.

\begin{defin}
For a logic $(\cL,\vDash_{\cL})$, a cardinal $\alpha$ that satisfies the properties of Proposition \ref{prop:closure-points} is called a \emph{closure point of $\cL$}.
\end{defin}

\begin{theorem}\label{thm:woodin-char}
The following are equivalent for a cardinal $\delta$.
\begin{enumerate}
	\item $\delta$ is Woodin.
	\item For every logic $(\cL, \vDash_\cL)$ with closure point $\delta$, there is $\kappa < \delta$ such that given any $\kappa<\lambda<\delta$ and any theory $T \subseteq \cL\cup \bL_{\kappa,\omega}(\tau)$ for a language $\tau\in V_\lambda$, if $T$ can be written as an increasing union $T=\bigcup_{\eta<\kappa}T_\eta$ of satisfiable theories, then $T$ has a transitive Henkin $\tau$-structure that is full up to $\lambda$ and $\cL$-correct up to $\lambda$.
\end{enumerate}
\end{theorem}

\begin{proof}
For the forward direction, fix a logic $\cL$ with a closure point at $\delta$. By definition the occurrence number of $\cL$, $o(\cL)<\delta$. For this proof, we need to consider another class related to $\cL$. Let $R$ be the class of quadruples $(\tau,\sigma,f,f_*)$ such that $\tau$ and $\sigma$ are languages, $f:\sigma\to\tau$ is a renaming and $f_*:\cL(\sigma)\to \cL(\tau)$ is the associated bijection. Using the fact that $\delta$ is Woodin, let $\kappa>o(\cL)$ be a ${<}\delta$-strong cardinal for the sets $\cL\upharpoonright V_\delta$, $\vDash_{\cL}\restriction V_\delta$, and $R\restriction V_\delta$.

First, observe that $\kappa$ must be a closure point of $\cL$. Suppose to the contrary that there is some $\sigma\in V_\kappa$ with $\cL(\sigma)\not\in V_\kappa$. Since $\delta$ is a closure point of $\cL$, there is $\lambda<\delta$ such that $\cL(\sigma)\in V_\lambda$. Let $j:V\to \cN$ be an elementary embedding with $\crit(j)=\kappa$, $j(\kappa)>\lambda$, and $V_\lambda\subseteq \cN$, and $\cL\restriction V_\lambda=j(\cL\restriction V_\lambda)\restriction V_\lambda$. Since $j(\sigma)=\sigma$ and $\cL\restriction V_\lambda=j(\cL\restriction V_\lambda)\restriction V_\lambda$, it follows that $j(\cL(\sigma))=\cL(\sigma)$. Since $\cL(\sigma)\not\in V_\kappa$, by elementarity, we have $j(\cL(\sigma))=\cL(\sigma)\not\in V_{j(\kappa)}$, which contradicts that $j(\kappa)>\lambda$. Thus, we have reached the desired contradiction showing that $\kappa$ is a closure point of $\cL$.

Fix a language $\tau\in V_\delta$ and a $\cL'=\cL\cup \bL_{\kappa,\omega}(\tau)$-theory $T=\bigcup_{\eta<\kappa}T_\eta$ as in the second clause. Note that, since a Woodin cardinal $\delta$ is Mahlo, $\delta$ is also a closure point of $\bL_{\kappa,\omega}$. Thus $\delta$ is a closure point of $\cL'$, and so the requirement that $\tau\in V_\delta$ implies that $T\in V_\delta$ as well. Fix $\lambda<\delta$ which is above the rank of $\tau$ and $\cL(\tau)$. Let $j:V\to \cN$ be an elementary embedding with
\begin{itemize}
\item $\crit j=\kappa$, $j(\kappa)>\lambda$,
\item $V_\lambda\subseteq \cN$,
\item $\cL\restriction V_\lambda=j(\cL\restriction V_\lambda)\restriction V_\lambda$
\item $(\vDash_{\cL})\restriction V_\lambda\times V_\lambda=j((\vDash_{\cL})\restriction V_\lambda\times V_\lambda)\restriction V_\lambda\times V_\lambda$.
\item $R\restriction V_\lambda^4=j(R)\restriction V_\lambda^4$.
\end{itemize}

Consider the sequence $j(\langle T_\eta\mid \eta<\kappa\rangle):=\langle T_\eta^*\mid \eta<j(\kappa)\rangle$. By elementarity,
$\cN$ thinks that $T_\kappa^*$ is a satisfiable $j(\cL')(j(\tau))$-theory,
and we have that $j\image T\subseteq T_\kappa^*$.

Let $M \in \cN$ be a $j(\cL')$-model of $T_\kappa^*$ in the sense of $\cN$. Let $\cM=V_\theta^{\cN}$ so that $\theta$ is a closure point of $j(\cL')$ and that $V_\theta^{\cN}\prec_{\Sigma_m} \cN$ for a very large $m$.  Define $\vDash^*$ to be the satisfaction relation $\vDash_{j(\cL')}$ from $\cN$ restricted to $V^{\cN}_\theta$.
We claim that $\hat{M}=(\cM,\in,\vDash^*, M)$ is a Henkin $\tau$-structure for $T$. It is easy to see that it satisfies clauses 1(a), 1(b) and 1(c) of Definition \ref{defn:henkin-structure}. It remains to verify clause (2).

Fix $\phi\in T$ and let $\tau_\phi$ be the fragment of $\cL(\tau)$ that occurs in $\phi$. Note that $\tau_\phi$ has size less than $\kappa$ by our assumption that $o(\cL)<\kappa$. We will argue that $M$ is a $\tau_\phi$-structure in $\cM$ via the renaming $f:\tau_\phi\to j(\tau_\phi)=j\image\tau_\phi$ defined by applying $j$ to $\tau_\phi$. Note that the renaming $f$ is an element of $\cN$, although the entire map $j:\tau\to j\image\tau$ might not be in $\cN$ (since the embedding is not a supercompactness embedding).\footnote{Note that technically the bijection $j:\tau\to j\image\tau$ is not a renaming in our sense because it is not an element of $\cN$, however it does have the required renaming properties because it is the union of a coherent collection of renamings that are elements of $\cN$.} It suffices then to argue that $f_*(\phi)=j(\phi)$ in $\cN$ since by elementarity, we know that in $\cN$, $M\vDash_{j(\mathcal L')}j(\phi)$. The argument that $f_*(\phi)=j(\phi)$ is straightforward for sentences from $\bL_{\kappa,\omega}(\tau)$, but is not obvious for sentences from $\cL(\tau)$ because we do not know how the satisfaction for $\cL$ works. So let us argue as follows.

Fix a language $\sigma\in V_\kappa$ and a renaming $g:\sigma\to\tau_\phi$. Since $\kappa$ is a closure point of $\cL$, we have $\cL(\sigma)\in V_\kappa$. Let $g_*(\bar\phi)=\phi$. By elementarity, $j(g):\sigma\to j(\tau_\phi)=j\image\tau_\phi$ is a renaming in $\cN$. It is easy to check that $j(g)=f\circ g$. Now, by elementarity, we have $$(j(g))_*(\bar\phi)=j(g_*)(\bar\phi)=j(g_*(\bar\phi))=j(\phi).$$ By our observations from Section~\ref{sec:prelim}, we have that $j(g)_*=(f\circ g)_*=f_*\circ (g_*)^{\cN}$, where $(g_*)^{\cN}$ is the unique bijection between $j(\cL)(\sigma)$ and $j(\cL)(\tau_\phi)$ as computed in $\cN$.  A priori, there is no reason to believe that the two maps which we both denoted by $g_*$ are the same.  However, we ensured that $R\restriction V_\lambda^4=j(R)\restriction V_\lambda^4$.  Since $g$ and $g_*$ are elements of $V_\lambda$, we have $f_*(\phi)=f_*(g_*(\bar \phi))=j(g)_*(\bar\phi)=j(\phi)$.


 Since $V_\lambda\subseteq M$ we have fullness up to $\lambda$ and by the properties of $j$, we have correctness up to $\lambda$ as well. \\

For the converse direction, suppose the second clause is true and fix any set $A\subseteq V_\delta$. We define the associated logic $\bL^A$ as follows. We extend $\bL^2$ by adding, for a binary relation $E$, a single new  formula $\phi_{A, E}(x)$, which holds whenever $E$ is well-founded on the transitive closure of $x$, $\text{tc}_E\,x$, and $\text{tc}_E\,x$ is isomorphic to $\text{tc}_\in\,a$ for some $a\in A$.

It is clear that $o(\bL^A) = \omega$ and $\delta$ is a closure point of $\bL^{A}$ because it was assumed to be a cardinal, so there is a cardinal $\kappa$ for $\bL^A$ as in the hypothesis, which we want to show is ${<}\delta$-strong for $A$.

Fix $\alpha>\kappa$ and fix $\lambda\gg\alpha$ below $\delta$.  Let $\tau$ be the language consisting of a binary relation $\in$ and constants $\{c_x\mid x\in V_{\alpha+1}\}\cup\{c\}$. Let $T$ be the $\bL^A\cup \bL_{\kappa,\omega}(\tau)$-theory $T$ consisting of the following sentences.
\begin{enumerate}
	\item ${\rm ED}_{\bL_{\kappa, \omega}}(V_{\alpha+1}, \in , c_x)_{x \in V_{\alpha+1}}$,
	\item $\{c_\xi \neq c< c_\kappa\mid \xi < \kappa\}$,\footnote{We use sentences $c_\xi\neq c$ instead of $c_\xi<c$ because it would be difficult to argue that the later are ${<}\kappa$-satisfiable without first verifying that $\kappa$ is regular.}
	\item Magidor's $\Phi$,
	\item $\forall x\,(x\in c_{A\cap V_\alpha}\rightarrow \phi_{A,\in}(x))$,
\item  $\forall x\, (\phi_{A,\in}(x)\wedge x\in c_{V_\alpha}\rightarrow x\in c_{A\cap V_\alpha})$.
	
\end{enumerate}

\noindent We can write $T$ as an increasing $\kappa$-union of satisfiable theories by filtrating the sentences in (2). Let $\hat{M} = (\cM,\in, \vDash^*,M)$ be a transitive Henkin $\tau$-structure for $T$ that is full up to $\lambda$ and $\bL^A$-correct up to $\lambda$.

Note once again that even though $\cM$ has only partial interpretations of $\tau$ on $M$, we get a total interpretation by unioning up the coherent interpretations from $\cM$ in $V$.

With the standard set-theoretic arguments, it suffices to produce an elementary embedding $j:V_{\alpha+1}\to M$ with $\crit(j)=\kappa$ , $V_\alpha\subseteq M$, $j(\kappa)>\alpha$ and $A\cap V_\alpha=j(A\cap V_\alpha)\cap V_\alpha$.

By the first clause in the definition of $T$, there is an elementary embedding $j:V_{\alpha+1} \to M$ (because first-order elementarity is absolute). The second clause guarantees that $j$ has a critical point $\leq\kappa$ and since every ordinal $\alpha<\kappa$ is definable in the logic $\bL_{\kappa,\omega}$, the critical point must be exactly $\kappa$. Note that the formulas used to ensure this are in $\bL_{\kappa, \omega}$, and thus reflected correctly in the transitive structure $\hat{M}$. Since $\hat M$ is transitive, it must be correct about $M$ being well-founded. So we can assume without loss that $M$ is transitive. Since $\hat{M}$ is full up to $\lambda$, it is full up to $\alpha$ and we have $V_\alpha \subseteq \cM$. But then since $\vDash^*$ in $\cM$ satisfies the basic properties of $\bL^2$, and it believes that $M\vDash^*\Phi$, it follows that $V_\alpha\subseteq M$.

Next, we will argue that for $a\in V_\alpha$, $M\vDash^* \phi_{A,\in}(a)$ if and only if $a\in A$. In $V$, the logic $\bL^A$ has the property that if $N$ is a well-founded model in $\bL^A$ and $\bar N$ is any transitive first-order submodel of $N$ such that for some transitive (from the perspective of $N$) set $B$, $B\subseteq \bar N$, then $\bar N$ and $N$ must agree on the formulas $\phi_{A,\in}(a)$ for $a\in B$. Thus, $\cM$ satisfies the same property of the relation $\vDash^*$. It follows, for our particular case, that $V_\alpha$ and $M$ agree on $\phi_{A,\in}(a)$ for $a\in V_\alpha$. But now the point is that $V_\alpha$ is small enough that the correctness of $\vDash^*$ up to $\lambda$ guarantees that $\phi_{A,\in}(a)$ holds true if and only if $a\in A$. Thus, $M$ must also be correct with respect to the formula $\phi_{A,\in}(a)$ for $a\in V_\alpha$. Since $M$ satisfies $\forall x\,(x\in c_{A\cap V_\alpha}\rightarrow \phi_{A,\in}(x))$, it follows that for $a\in V_\alpha$ and $a\in j(A\cap V_\alpha)$, we have $a\in A$. Thus, $j(A)\cap V_\alpha\subseteq A\cap V_\alpha$. Since $M$ satisfies $\forall x\, (\phi_{A,\in}(x)\wedge x\in c_{V_\alpha}\rightarrow x\in c_{A\cap V_\alpha})$, it follows that for $a\in V_\alpha$ and $a\in A$, we have $a\in j(A\cap V_\alpha)$. Thus, $A\cap V_\alpha\subseteq j(A)\cap V_\alpha$, which gives the desired equality.


\end{proof}

This sort of characterization also suggests a new hierarchy of Woodin cardinals based on the logics that they endow non-standard chain compactness to.  For instance, using V\"{a}\"{a}n\"{a}nen's sort logics, we get definable notions of Woodin cardinals.

\begin{theorem}\label{th:externallyDefinableWoodinCharacterization} The following are equivalent for a cardinal $\delta$.
\begin{enumerate}
\item $\delta$ is externally definable Woodin.
\item For every $n<\omega$ and $a\in V_\delta$, there is a cardinal $\kappa<\delta$ above $\rank(a)$ such that for every $\kappa < \lambda < \delta$ and theory $T \subseteq \bL^{s, \Sigma_n}_{\kappa,\omega}(\tau)$ for a language $\tau\in V_\lambda$, if $T$ can be written as an increasing union $T=\bigcup_{\eta<\kappa}T_\eta$ of satisfiable theories, then $T$ has a transitive $n$-correct for $V_\lambda$ Henkin $\tau$-structure  that is full up to $\lambda$.
\end{enumerate}
\end{theorem}

\begin{proof}
Suppose $\delta$ is externally definable Woodin.  Fix $n<\omega$.  Let $$A=\{(\ulcorner\varphi(x)\urcorner,b)\mid \varphi(x)\text{ is }\Sigma_n\text{ and }V\vDash\varphi(b)\},$$
and note that $A$ is definable using the $\Sigma_n$-truth predicate.  Fix a language $\tau$ and let $\kappa>\rank(\tau),\rank(a)$ be as in the definition of externally definable Woodin cardinals for the definable class $A$. Fix a $\bL^{s,\Sigma_n}_{\kappa,\omega}(\tau)$-theory $T$ as in (2). Pick a limit $\lambda>\kappa$ below $\delta$ such that $T\in V_\lambda$. Let \hbox{$j:V\to \cN$} be an elementary embedding with $\crit j=\kappa$, $j(\kappa)>\lambda$, $V_\lambda\subseteq \cN$, and $A\cap V_\lambda=j(A)\cap V_\lambda$.

Now consider the sequence $j(\langle T_\eta\mid \eta<\kappa\rangle)=\langle T_\eta^*\mid \eta<j(\kappa)\rangle$. By elementarity, it is an increasing sequence, and hence, $T_\kappa^*$ is a theory containing $j\image T$. So, by elementarity, $\cN$ thinks that $T_\kappa^*$ has a model $M$ in $\bL^{s,\Sigma_n}_{\kappa,\omega}$ (technically it would have to be $\bL^{s,\Sigma_n}_{j(\kappa),\omega}$, but we can always prune to get rid of the extra assertions). Let $\cM=V_\theta^{\cN}\prec_{\Sigma_m} \cN$ for some large enough $\theta$ and $m>n$, and let $\vDash^*$ be the satisfaction $\vDash_{\bL^{s,\Sigma_n}_{\kappa,\omega}}$ of $\cN$ restricted to $V^{\cN}_\theta$. We claim that $\hat M=(V_\theta^\cN,\in,\vDash^*,M)$ is an $n$-correct for $V_\lambda$ Henkin $\tau$-structure for $T$.

The argument that $\hat M$ is a Henkin $\tau$-structure for $T$ is even easier than in the proof of Theorem~\ref{thm:woodin-char} because we know exactly what the formulas in $\bL^{s,\Sigma_n}_{\kappa,\omega}$ look like. So it remains to check $n$-correctness for $V_\lambda$. Since $A$ is definable and $j$ is elementary, we have
$$j(A) = \{(\ulcorner\varphi(x)\urcorner,b)\mid\varphi(x)\text{ is }\Sigma_n\text{ and }\cN\vDash\varphi(b)\}$$
Also, since $\lambda$ is limit, we have that
$$A \cap V_\lambda =\{(\ulcorner\varphi(x)\urcorner,b)\mid b\in V_\lambda,\, \varphi(x)\text{ is }\Sigma_n\text{ and }V\vDash\varphi(b)\}$$
and similarly for $j(A) \cap V_\lambda$. By our assumptions on $j$, $A\cap V_\lambda=j(A)\cap V_\lambda$. But this means precisely that for $a\in V_\lambda$, we have $V\vDash\varphi(a)$ if and only if $\cN\vDash\varphi(a)$ if and only if $V_\theta^{\cN}\vDash\varphi(a)$.

For the converse direction, we fix a $\Sigma_n$-formula $\varphi(x,y)$ and a parameter $a\in V_\delta$.  Let $\kappa<\delta$ be as in the statement of (2) for $n$ and $a$.  Fix $\kappa < \alpha < \delta$.  We adapt the corresponding argument from Theorem \ref{thm:woodin-char}. We replace the assertion $\phi_{A,\in}(x)$ of that proof with the formula $\phi_{\varphi,E}(x)$ asserting that the transitive closure of $x$ is isomorphic to the transitive closure of an element $b$ such that $\phi(b,a)$. To express $\phi_{\varphi,E}(x)$ in the logic $\bL^{s,\Sigma_{n}}_{\kappa,\omega}$ we use that $\varphi(x,y)$ is $\Sigma_n$ and the parameter $a$ is definable using an infinitary atomic formula since it has rank below $\kappa$. So we can let $T$ be the theory as in that proof, but expressed using the sort logic $\bL^{s, \Sigma_{n}}_{\kappa,\omega}$.  By assumption $T$ must have an $n$-correct for $V_\lambda$ Henkin $\tau$-structure $\hat{M}=(\cM,\in,\vDash^*,M)$ that is full up to $\lambda$ for $\lambda\gg\alpha$.   As in that proof we get an elementary embedding $j:V_{\alpha+1}\to M$ with $\crit j=\kappa$.  So it remains to verify that $M\vDash^*\phi_{\varphi,\in}(b)$ for $b\in V_\alpha$ if and only if $\varphi(b,a)$ holds in $V$. Since by $n$-correcteness, $\cM$ is correct about the $\Sigma_{n}$ properties of $\vDash^*$, it recognizes that $\vDash^*$ is satisfaction for $\bL^{s,\Sigma_n}_{\kappa,\omega}$. Thus, $\cM$ believes that $M\vDash^*\phi_{\varphi,\in}(b)$ if and only if $\varphi(b,a)$ holds, and again, by $n$-correctness, $\cM$ is correct about $\varphi(b,a)$ for $b\in V_\alpha$.

\end{proof}
Note that the proof of Theorem~\ref{th:externallyDefinableWoodinCharacterization} shows that we could have additionally required the Henkin $\tau$-structure of Theorem~\ref{thm:woodin-char} to be $n$-correct for a specified $n$.

Next, we observe that if we relax clause (2) of Theorem~\ref{thm:woodin-char} to just ${<}\kappa$-satisfiable theories, then we get a characterisation of Woodin for strong compactness cardinals.

\begin{theorem}\label{thm:wfsc-char}
The following are equivalent for a cardinal $\delta$.
\begin{enumerate}
	\item $\delta$ is Woodin for strong compactness.
	\item For every logic $(\cL, \vDash_\cL)$ with closure point $\delta$,there is $\kappa < \delta$ such that given any $\kappa < \lambda < \delta$ and theory $T \subseteq \cL \cup \bL_{\kappa, \omega}(\tau)$ with $\tau\in V_\lambda$, if $T$ is ${<}\kappa$-satisfiable, then $T$ has a transitive Henkin $\tau$-structure that is full up to $\lambda$ and $\cL$-correct up to $\lambda$.
\end{enumerate}
\end{theorem}

\begin{proof}
For (1) implies (2), the only difference from Theorem \ref{thm:woodin-char}, is that the embedding $j:V\to\cN$ that we use can be assumed to satisfy the strong compactness $\lambda$-covering property. Thus, there is a set $s\in\cN$ such that $j\image \lambda\subseteq s$ and $|s|^{\cN}<j(\kappa)$. It follows that there is a set $t\in \cN$ such that $j\image T\subseteq t$ and $|t|^{\cN}<j(\kappa)$. Then, $t\cap j(T)$ is a $j(\cL)$-theory containing $j\image T$ and we use the fact that by elementarity, $j(T)$ is ${<} j(\kappa)$-satisfiable in the sense of $\cN$.

For (2) implies (1), if we fix $A\subseteq V_\delta$, the same proof as for Theorem \ref{thm:woodin-char} can be used to show that there is $\kappa<\delta$ which is ${<}\delta$-strong for $A$. We fix $\kappa<\alpha\ll\lambda<\delta$. We augment our language there by a constant $s$ and augment our theory $T$ to contain statements $\{c_\xi\in s\mid \xi<\lambda\}$ and the statement $|s|<c_\kappa$. The theory $T$ is clearly ${<}\kappa$-satisfiable since we can always satisfy ${<}\kappa$-many statements $c_\xi\in s$ together with $|s|<c_\kappa$. Clearly the embedding $j:V_{\alpha+1}\to \cM$ is then a $\lambda$-strong compactness embedding as witnessed by the interpretation of $s$.

\end{proof}

We could ask about pushing these results higher, mixing characterizations from \cite{m-roleof,b-sccomp,b-mtlc} with our new notion of abstract Henkin structures to find model-theoretic characterizations for Woodin for supercompactness (see \cite{as-universal,a-sargsyan}) or prospective notions of Woodin for extendibility.  However, Perlmutter \cite[Theorem 5.10]{p-sc-ah} has shown that Woodin for supercompactness is equivalent to being a Vop\v{e}nka cardinal, which has a compactness characterization due to Makowsky \cite[Theorem 2]{m-vopcomp}, namely, $\kappa$ is Vop\v{e}nka if and only if  $V_\kappa \vDash \text{``Every abstract logic has a strong compactness cardinal"}$.

\section{Virtual large cardinals}\label{sec:virtual}
In this section, we will give compactness characterizations for various virtual large cardinals using a new notion of pseudo-models.

Sometimes to prove that a cardinal $\kappa$ is some virtual large cardinal from a given compactness assumption, we will need to argue that the compactness assumption yields the existence of a \emph{virtual $\cL$-embedding} $j:M\to N$ (with $M,N\in V$) for an abstract logic $\cL$. Let us explain what we mean by virtual $\cL$-embeddings. Having fixed a language $\tau$, we have that both $\cL(\tau)$ and $\vDash_{\cL}$ (restricted to $M\cup N$) are sets. Using these two actual sets from $V$, we can interpret $\cL$ (restricted to these models) in a forcing extension of $V$ de re (using the sets from V) and not de dicto (using its definition from $V$ as interpreted by the forcing extension). In this way, a virtual $\cL$-elementary embedding $j:M\to N$ has the property that for every formula $\varphi(x)$ in the set $\cL(\tau)$ and $a\in M$, $M$ satisfies $\varphi(a)$ according to $\vDash_{\cL}$ if and only if $N$ satisfies $\varphi(j(a))$ according to $\vDash_{\cL}$. We should note here that de dicto interpretations of an abstract logic in a forcing extension, using its $V$-definition, might yield a logic with entirely different properties. For instance the logic $\bL_{\kappa,\kappa}$ for a weakly compact $\kappa$ defined using the parameter $\kappa$ in $V$ will turn into the logic $\bL_{\omega_1,\omega_1}$ in a forcing extension by the L\'{e}vy collapse ${\rm Coll}(\omega,{<}\kappa)$.

The following simple observation about virtual embeddings will be used repeatedly.
\begin{obs}\label{obs:kunen}
Suppose $j:V_\alpha\to M$ is a virtual embedding with $\crit j=\kappa$ and $V_\alpha$ has a bijection $f:\kappa\to A$. Then $j\image A\in V$.
\end{obs}
\begin{proof}
For $a\in A$, given $f(\xi)=a$, we have $j(a)=j(f)^{-1}(\xi)$. Since $M\in V$ by assumption, we have both $f$ and $j(f)$ in $V$, and therefore we can recover $j\image A$.
\end{proof}
We start with virtually extendible cardinals. Magidor showed\footnote{Magidor only explicitly dealt with the first strong compactness cardinal of $\bL^2$, but the arguments easily give what is claimed here.} that extendible cardinals $\kappa$ are precisely the strong compactness cardinals for the second-order infinitary logic $\bL^2_{\kappa,\kappa}$ \cite{m-roleof}. Indeed, his arguments show that if $\kappa$ is a chain compactness cardinal for $\bL^2_{\kappa,\kappa}$, then $\kappa$ is extendible, and hence, we have the following equivalence.
\begin{theorem}[Magidor] The following are equivalent for a cardinal $\kappa$.
\begin{enumerate}
\item $\kappa$ is extendible.
\item $\kappa$ is a strong compactness cardinal for $\bL^2_{\kappa,\kappa}$.
\item $\kappa$ is a chain compactness cardinal for $\bL^2_{\kappa,\kappa}$.
\end{enumerate}
\end{theorem}

We will adapt Magidor's characterization to virtually extendible cardinals using a new notion of pseudo-models. The notion of a ``forth system" below is simply the forward direction of back-and-forth systems from model theory, see Definition \ref{game-def}.

\begin{defin} \label{forth-def}
Let $\mathcal L$ be a logic, and $\tau$ and $\tau^*$ be languages. Let $T$ be an $\cL(\tau)$-theory, $\cM$ be a $\tau^*$-structure, and $\delta$ be a cardinal.
\begin{enumerate}
	\item A \emph{$\delta$-forth system} $\cF$ from $\tau$ to $\tau^*$ is a collection of renamings $f:\sigma\to\sigma^*$ with $\sigma\in \cP_\delta \tau$ and $\sigma^*\in\cP_\delta \tau^*$ satisfying the following properties.
	\begin{enumerate}
		\item $\emptyset \in \cF$.
		\item If $f \in \cF$ and $\tau_0 \in \cP_\delta \tau$, then there is $g \in \cF$ with $f \subseteq g$ and $\tau_0 \subseteq \dom g$.
	\end{enumerate}
	\item $\cM$ is a \emph{$\delta$-pseudo-model for $T$} if there is a $\delta$-forth system $\cF$ from $\tau$ to $\tau^*$ such that for every $f \in \cF$, $\cM$ is a model of $f_*\image (T\cap \cL(\dom f)$).
\end{enumerate}
\end{defin}
\noindent We will call $\omega$-forth systems and $\omega$-pseudo-models simply forth systems and pseudo-models respectively.  

Cleary, any actual model is a pseudo-model as well, but (as seen in the following example) there are unsatisfiable theories with a $\omega_1$-pseudo-model (and this example can be generalized).

\begin{example}
Fix the language $\tau = \{R, c_\alpha : \alpha < \omega_1\}$ with $R$ a unary predicate, and let $Q$ be the quantifier `there exists uncountably many.'  Consider the theory
$$T=\left\{\neg Qx R(x), R(c_\alpha), c_\alpha \neq c_\beta: \alpha < \beta < \omega_1\right\}$$
This theory demands that $R$ is both uncountable and countable, so is unsatisfiable.  To find an $\omega_1$-pseudo model, set $\tau_* = \{S, d_n :n < \omega\}$ and $M = \seq{\omega_1; \omega, n}_{n<\omega}$.  To define the forth system $\cF$, take any $\sigma \in \cP_{\omega_1}\tau$ (written as $\{R, c_\alpha : \alpha \in S_\sigma\}$ for countable $S_\sigma \subset \omega_1$) and injection $\pi:S_\sigma \to \omega$ so $\im \pi$ is countable and cocountable.  We can define $f^{\sigma, \pi}:\sigma \to \tau_*$ by $f^{\sigma, \pi}(R)=S$ and $f^{\sigma, \pi}(c_\alpha) = d_{\pi(\alpha)}$.  Then the collection of all such $f^{\sigma, \pi}$ is a forth system from $\tau$ to $\tau_*$ (the cocountability of $\im \pi$ is key).  Further, 
$${f^{\sigma,\pi}_*}''\left(T \cap \bL(Q)(\sigma)\right) = \left\{\neg Qx R(x), R(d_n), d_n\neq d_m : n < m \in \im \pi\right\}$$
Then $M$ models this.
\end{example}

\begin{defin}
Let $\cL$ be a logic and $\kappa$, $\delta$ be cardinals.
\begin{enumerate}
\item $\kappa$ is a \emph{$\delta$-pseudo-compactness cardinal for $\cL$} if every ${<}\kappa$-satisfiable $\cL$-theory has a $\delta$-pseudo model.
\item $\kappa$ is a \emph{$\delta$-pseudo-chain compactness cardinal for $\cL$} if every $\cL$-theory $T$, which can be written as an increasing union $T=\bigcup_{\eta<\kappa}T_\eta$ of satisfiable theories, has a $\delta$-pseudo-model.
\end{enumerate}

\end{defin}
Note that a $\kappa^+$-pseudo-chain compactness cardinal for a logic $\cL$ is a weak compactness cardinal for $\cL$. We will see below that the converse fails to hold.

\begin{theorem}\label{ve-thm}
The following are equivalent for a cardinal $\kappa$.
\begin{enumerate}
\item $\kappa$ is virtually extendible.
\item $\kappa$ is a $\kappa^+$-pseudo-compactness cardinal for $\bL^2_{\kappa,\kappa}$.
\item $\kappa$ is an $\omega$-pseudo-compactness cardinal for $\bL^2_{\kappa,\kappa}$.
\item $\kappa$ is an $\omega$-pseudo-compactness cardinal for $\bL^2 \cup \bL_{\kappa, \omega}$.
\end{enumerate}
\end{theorem}
\begin{proof} Assume that $\kappa$ is an $\omega$-pseudo-compactness cardinal for $\bL^2\cup \bL_{\kappa,\omega}$ and fix $\alpha > \kappa$. Let $\tau$ be the language consisting of a binary relation $\in$ and constants $\{c_x\mid x\in V_\alpha\}\cup \{d_\xi\mid \xi\leq\alpha\}\cup\{c\}$. Let $T$ be the following $\bL^2\cup\bL_{\kappa, \omega}(\tau)$-theory:
$${\rm ED}_{\bL_{\kappa, \omega}}(V_\alpha, \in, c_x)_{x\in V_\alpha} \cup \{c_\xi \neq c < c_\kappa \mid \xi < \kappa\} \cup \{\Phi\} \cup \{d_\xi < d_\eta < c_\kappa \mid \xi < \eta \leq \alpha\}, $$ where ${\rm ED}$ stands for elementary diagram, each constant $c_x$ is interpreted as $x$ in $V_\alpha$, and $\Phi$ is Magidor's $\bL^2(\{\in\})$-sentence encoding
the well-foundedness of $\in$ and that the model is isomorphic to some $V_\beta$. Clearly $V_\alpha$ satisfies every piece of $T$ of size less than $\kappa$.

By assumption, there is some $\tau^*$-structure $\cM$ and a forth system $\cF$ witnessing that $\cM$ is a pseudo-model for $T$.  We can fix a way of renaming $\in$ and look at the forth system extending this renaming, so that without loss, we assume that $\tau$ and $\tau^*$ use the same symbol for $\in$.  In particular, the model $\cM$ satisfies $\Phi$, which means it is well-founded and the Mostowski collapse gives $\pi:\cM\cong V_\beta$ for some $\beta$.  So  $V_\beta$ is a pseudo-model for $T$.  The witnessing forth system under the inclusion ordering is a poset $\mathbb P$.  Forcing with $\mathbb P$ yields a bijection\footnote{Although a bijection between languages coming from outside the universe $V$ may not satisfy the properties of a renaming with respect to a particular logic from $V$, the bijection $f$ does because it is a union of renamings from $V$.} $f:\tau \to \tau^*$ in the forcing extension that is a union of the generically chosen collection of renamings from $\cF$. The bijection $f$ allows us to build a virtual $\bL_{\kappa,\omega}$-elementary embedding $j:V_\alpha\to V_\beta$.  We also need to verify that $\crit j = \kappa$ and $j(\kappa) > \alpha$. Since every ordinal $\xi<\kappa$ is definable in the logic $\bL_{\kappa,\omega}$ in any transitive model of set theory of which it is an element (argue by induction on $\xi<\kappa$), the critical point of $j$ must be $\kappa$. The inclusion of the $d_\xi$ constants forces there to be ordinals of order-type $\alpha$ below $j(\kappa)$, so $j(\kappa) > \alpha$. Thus, $(4)\rightarrow (1)$.

Now suppose that $\kappa$ is virtually extendible.  Given a ${<}\kappa$-satisfiable $\bL^2_{\kappa,\kappa}(\tau)$-theory $T$, let $F:\cP_\kappa T \to Str\, \tau$ be a map such that $F(s)\vDash s$.  Let $\alpha$ be large enough that $V_\alpha$ contains $F$ and witnesses this property.  Using virtual extendibility, in a forcing extension $V[G]$ by ${\rm Coll}(\omega,V_\alpha)$, there an elementary embedding $j:V_\alpha \to V_\beta$  with $\crit j=\kappa$ and $j(\kappa) > \alpha$. Since $j(\kappa)$ is inaccessible by elementarity, it follows that $|V_\alpha|<j(\kappa)$. Since the forcing ${\rm Coll}(\omega,V_\alpha)$ has size $|V_\alpha|$, we can cover $j\image T$ by a set $Y$ in $V$ of size $<j(\kappa)$. We will argue that the $\tau^*=j(\tau)$-structure $\cM=j(F)(Y)$ is a $\kappa^+$-pseudo-model of $T$.

What we would like to say is that $\cF$ is the collection of all $f\subseteq j\rest \tau$ of size $\kappa$, but we cannot do this because we do not have access to $j$ in $V$. However, we can do something relatively close using the forcing relation. Let $\dot j$ be a ${\rm Coll}(\omega,V_\alpha)$-name that is forced to be a virtual extendibility embedding from $V_\alpha$ to $V_\beta$. Define $\cF$ to be the collection of all renamings $f:\sigma\to\sigma^*$ with $\sigma\in\cP_{\kappa^+} \tau$ and $\sigma^*\in\cP_{\kappa^+} \tau^*$ such that there is a condition $p\in{\rm Coll}(\omega,V_\alpha)$ with $$p\Vdash\check f=\dot j\rest \check\sigma:\check\sigma\to \dot j\image\check\sigma.$$ The system $\cF$ is non-empty by Observation~\ref{obs:kunen}. It has the extension property because whenever a condition $p\Vdash\check f=\dot j\rest \check\sigma:\check\sigma\to j\image\check\sigma$ and there is some $\tau_0\in \cP_{\kappa^+}\tau$, then there is a condition $q\leq p$ deciding that some $g$ is a renaming between $\sigma\cup \tau_0$ and $j\image (\sigma\cup\tau_0)$, and since $q\leq p$, $g$ must extend $f$. Thus, $(1)\rightarrow (2)$.

The rest of the implications are trivial.
\end{proof}
The above proof would not go through with just a $\kappa^+$-pseudo-chain compactness cardinal $\kappa$ because we cannot filtrate the part of the theory $T$ which involves the constants $d_\xi$. We will show below that the $\kappa^+$-chain compactness cardinals $\kappa$ for $\bL^2_{\kappa,\kappa}$ are precisely the weakly virtually extendible cardinals.

\begin{cor}
Every virtually extendible cardinal $\kappa$ is a weak compactness cardinal for $\bL^2_{\kappa,\kappa}$.
\end{cor}

The converse fails to hold. Let us observe here that much weaker large cardinals $\kappa$ can be weak compactness cardinals for $\bL^2_{\kappa,\kappa}$. Hamkins and Johnstone defined that an inaccessible cardinal $\kappa$ is \emph{strongly uplifting} if for every $A\subseteq\kappa$, there are arbitrarily large regular $\theta>\kappa$ such that \hbox{$(V_\kappa,\in,A)\prec (V_\theta,\in,\bar A)$} for some $\bar A\subseteq V_\theta$ \cite{HamkinsJohnstone2017:StronglyUpliftingCardinalsAndBoldfaceResurrection}. Strongly uplifting cardinals are weaker than subtle cardinals in strength, and hence in, particular, much weaker than virtually extendible cardinals. These cardinals are also weak compactness cardinals for $\bL^2_{\kappa,\kappa}$, but are still not the optimal bound.
\begin{prop}
Every strongly uplifting cardinal $\kappa$ is a weak compactness cardinal for $\bL^2_{\kappa,\kappa}$ and it has below it a cardinal $\delta$ that is a weak compactness cardinal for $\bL^2_{\delta,\delta}$.
\end{prop}
\begin{proof}
First, suppose that $\kappa$ is strongly uplifting.  Given a ${<}\kappa$-satisfiable theory $T$ in $\bL^2_{\kappa,\kappa}(\tau)$ of size $\kappa$, we can assume without loss that $T\subseteq V_\kappa$. Choose a large enough $\theta$ and find $\bar T \subseteq V_\theta$ such that
$$(V_\kappa,\in,T)\prec (V_\theta,\in,\bar T)$$
and $V_\theta$ sees that every proper initial segment of $T$ has a model. By elementarity, $(V_\kappa,\in,T)$ must then also satisfy that every proper initial segment of $T$ has a model, but then again by elementarity, we get that $(V_\theta,\in,\bar T)$ satisfies that every proper initial segment of $\bar T$, in particular $T$,  has a model.

Now choose any $\theta$ such that $V_\kappa\prec V_\theta$. The model $V_\theta$ satisfies that $\kappa$ is a weak compactness cardinal for $L^2_{\kappa,\kappa}$. Hence, by elementarity, $V_\kappa$ satisfies that there is a weak compactness cardinal $\delta$ for $L^2_{\delta,\delta}$. Let's argue that $V_\kappa$ is correct about $\delta$. Fix a ${<}\delta$-satisfiable theory $T$ of size $\delta$. We can assume without loss of generality that $T\in V_\kappa$. Choose a large enough $V_\rho$ which sees that $T$ is ${<}\delta$-satisfiable. By elementarity, $V_\kappa$ then also satisfies that $T$ is ${<}\delta$-satisfiable. Hence $V_\kappa$ has a model of $T$.
\end{proof}

Magidor showed that the least cardinal $\kappa$ which is a strong compactness cardinal for $\bL^2$ must be extendible \cite{m-roleof}. We get an interesting not quite analogous reformulation with virtual extendibility and pseudo-models.
\begin{theorem}\label{th:leastCompactness} $\,$
\begin{enumerate}
\item Suppose $\kappa$ is the least $\kappa^+$-pseudo-compactness cardinal for $\bL^2$. Then either $\kappa$ is virtually extendible or there is a measurable cardinal below it.
\item Suppose the ${\rm GCH}$ holds and there are no measurable cardinals. Then $\kappa$ is virtually extendible if and only if it is a $\kappa^+$-pseudo-compactness cardinal for $\bL^2$.
\end{enumerate}
\end{theorem}
\begin{proof}
We will only prove (1) because (2) follows from the argument. Suppose $\kappa$ is the least $\kappa^+$-pseudo-compactness cardinal for $\bL^2$ and fix a strong limit cardinal $\alpha>\kappa$ of countable cofinality. Let the language $\tau$ and theory $T$ be as in the proof of Theorem~\ref{ve-thm}, with the only difference that we take the elementary diagram of $(V_\alpha,\in,c_x)_{x\in V_\alpha}$ in $\bL^2$ as opposed to $\bL^2_{\kappa,\omega}$. By the compactness assumption, there is $\beta>\kappa$ such that $V_\beta$ is a $\kappa^+$-pseudo-model of $T$. So there is a virtual elementary embedding $j:V_\alpha\to V_\beta$ with $\crit j=\gamma\leq\kappa$ and $j(\kappa)>\alpha$ whose $\kappa$-sized pieces are in $V$. First, suppose that $2^\gamma\leq\kappa$. In this case, because we have $\kappa$-sized pieces of $j$, it follows that $\gamma$ is measurable. So suppose that $2^\gamma>\kappa$. It follows, by elementarity, that $2^{j(\gamma)}>j(\kappa)$. But since we assumed that $\alpha$ is a strong limit and $j(\kappa)>\alpha$, it must be that $j(\gamma)>\alpha$ as well (note that $j(\gamma)$ is inaccessible and therefore cannot be $\alpha$).

Assuming that there are no measurable cardinals below $\kappa$, by the pigeon-hole principle, there must be some critical point $\gamma\leq\kappa$, which works for unboundedly many ordinals $\alpha$. So $\gamma$ is virtually extendible, which means that it cannot be below $\kappa$ by the leastness assumption.
\end{proof}
We do not know whether an analogous result holds $\omega$-pseudo-compactness cardinals for $\bL^2$. More generally, we do not have any results separating $\omega$-pseudo-compactness cardinals and $\kappa^+$-pseudo compactness cardinals $\kappa$.
\begin{question}
For some logic $\cL$ and language $\tau$, is there an uncountable theory $T$ in $\cL(\tau)$ which has a pseudo-model, but not an $\omega_1$-pseudo-model?
\end{question}
\begin{question}
Is there a logic $\cL$ and a cardinal $\kappa$ which is an $\omega$-pseudo-compactness cardinal for $\cL$, but not a $\kappa^+$-pseudo-compactness cardinal for $\cL$?
\end{question}
Next, we show that weakly virtually extendible cardinals $\kappa$ are precisely the $\kappa^+$-pseudo-chain compactness cardinals for $\bL^2_{\kappa,\kappa}$.
\begin{theorem}\label{wve-thm}
The following are equivalent for a cardinal $\kappa$.
\begin{enumerate}
\item $\kappa$ is weakly virtually extendible.
\item $\kappa$ is a $\kappa^+$-chain-compactness cardinal for $\bL^2_{\kappa,\kappa}$.
\item $\kappa$ is an $\omega$-chain-compactness cardinal for $\bL^2_{\kappa,\kappa}$.
\item $\kappa$ is an $\omega$-chain-compactness cardinal for $\bL^2\cup \bL_{\kappa,\omega}$.
\end{enumerate}
\end{theorem}
\begin{proof}
Assume that $\kappa$ is an $\omega$-chain-compactness cardinal for $\bL^2\cup \bL_{\kappa,\omega}$ and fix $\alpha > \kappa$. Let $\bar\tau$ be the language from the proof of Theorem~\ref{ve-thm} without the constants $d_\xi$ and let $\bar T$ be the theory $T$ without the statements about the $d_\xi$.
We can filtrate $\bar T$ by including in $\bar T_\eta$ the statements $\{c_\xi \neq c < c_\kappa \mid \xi < \eta\}$. By hypothesis, some $V_\beta$ is a pseudo-model for the theory $\bar T$, and this yields a virtual elementary embedding $j:V_\alpha\to V_\beta$ with $\crit j=\kappa$.

Now suppose that $\kappa$ is weakly virtually extendible and $T=\bigcup_{\eta<\kappa}T_\eta$ is a  $\bL^2_{\kappa, \kappa}(\tau)$-theory such that $\vec T=\langle T_\eta\mid\eta<\kappa\rangle$ is an increasing sequence of satisfiable theories. Let $\alpha$ be large enough so that $V_\alpha$ witnesses all this and let $j:V_\alpha\to V_\beta$ with $\crit j =\kappa$ be a virtual elementary embedding. By elementarity, $V_\beta$ satisfies that $j(\vec T)$ is an increasing sequence of theories and for all $\eta<j(\kappa)$, $j(\vec T)(\eta)$ is satisfiable, so in particular, $V_\beta$ has a model $\cN\vDash j(T)(\kappa)\supseteq T$. The model $\cN$ is the required $\kappa^+$-pseudo-model for $T$.

\end{proof}
It follows, in particular, that weakly virtually extendible cardinals $\kappa$ are also weak compactness cardinals for $\bL^2_{\kappa,\kappa}$.
\begin{theorem}
Suppose $\kappa$ is the least $\omega$-chain-compactness cardinal for $\bL^2$. Then $\kappa$ is weakly virtually extendible.
\end{theorem}
\begin{proof}
Fix $\alpha>\kappa$. Let the language $\bar\tau$ and theory $\bar T$ be as in the proof of Theorem~\ref{wve-thm}, with the only difference that we take the elementary diagram of $(V_\alpha,\in,c_x)_{x\in V_\alpha}$ in $\bL^2$ as opposed to $\bL^2_{\kappa,\omega}$. By the compactness assumption, there is a virtual elementary embedding $j:V_\alpha\to V_\beta$ with $\crit j=\gamma\leq\kappa$. By the pingeon-hole principle, there is some $\gamma\leq\kappa$ that works for unboundedly many $\alpha$. But then $\gamma$ must be weakly virtually extendible and so $\gamma=\kappa$ by the leastness assumption.
\end{proof}

We will now give compactness characterizations of several other virtual large cardinal notions by reformulating the known compactness properties of the original large cardinals in terms of pseudo models. At the same time, we will see that such a translation fails to hold for the virtual \Vopenka's principle as a consequence of the splitting of virtual $C^{(n)}$-extendibility into the weak and strong forms.

It is a folklore result that measurable cardinals are precisely the chain compactness cardinals for $\bL_{\kappa,\kappa}$ (see for instance, \cite{changkeisler}, Exercise 4.2.6) .
\begin{theorem}
The following are equivalent.
\begin{enumerate}
\item $\kappa$ is virtually measurable.
\item $\kappa$ is a $\kappa^+$-pseudo-chain compactness cardinal for $\bL_{\kappa,\kappa}$.
\item $\kappa$ is an $\omega$-pseudo-chain compactness cardinal for $\bL_{\kappa,\kappa}$.
\end{enumerate}
\end{theorem}
\begin{proof}
Assume that $\kappa$ is an $\omega$-pseudo-chain compactness cardinal for $\bL_{\kappa,\kappa}$ and fix $\alpha>\kappa$ with $\text(cof)(\alpha)>\kappa$. Let $\tau$ be the language consisting of a binary relation $\in$ and constants $\{c_x\mid x\in V_\alpha\}\cup \{c\}$. Let $T$ be the following $\bL_{\kappa,\kappa}(\tau)$-theory:
$${\rm ED}_{\bL_{\kappa,\kappa}}(V_\alpha, \in,c_x)_{x\in V_\alpha}\cup \{c_\xi\neq c<c_\kappa\mid \xi<\kappa\},$$
where each constant $c_x$ is interpreted as $x$.  We can filtrate $\bar T$ by including in $T_\eta$ the statements $\{c_\xi \neq c < c_\kappa \mid \xi < \eta\}$. By assumption, $T$ has a pseudo-model $\cM$. The model $\cM$ must be well-founded because well-foundedness is expressible in $\bL_{\kappa,\kappa}$, and so this assertion must be contained in the elementary diagram of $V_\alpha$. Next, note $\cM$ is closed under ${<}\kappa$-sequences because ${<}\kappa$-closure is  expressible in $\bL_{\kappa,\kappa}$ and $V_\alpha^{{<}\kappa}\subseteq V_\alpha$ by our choice of $\alpha$. Finally, there is a virtual elementary embedding $j:V_\alpha\to \cM$ with $\crit j=\kappa$ (since every ordinal below $\kappa$ is $\bL_{\kappa,\kappa}$-definable).

In the other direction, suppose that $\kappa$ is virtually measurable and $T=\bigcup_{\eta<\kappa}T_\eta$ is a  $\bL_{\kappa, \kappa}(\tau)$-theory such that $\vec T=\langle T_\eta\mid\eta<\kappa\rangle$ is an increasing sequence of satisfiable theories. Let $\alpha$ be large enough so that $V_\alpha$ witnesses all this and let $j:V_\alpha\to \cM$ with $\crit j =\kappa$ be a virtual elementary embedding. By elementarity, $\M$ satisfies that $j(\vec T)$ is an increasing sequence of theories and for all $\eta<j(\kappa)$, $j(\vec T)(\eta)$ is satisfiable, so in particular, it has a model $\cN\models j(T)(\kappa)\supseteq j\image T$. Since $j\image T\subseteq \bL_{\kappa,\kappa}(j(\tau))$ and $\cM$ is closed under sequences of length less than $\kappa$, it is correct about $\cN$ being a model of $j\image T$. The model $\cN$ is the required $\kappa^+$-pseudo-model for $T$.
\end{proof}

Benda \cite{b-sccomp} provided a compactness characterization of supercompact cardinals in terms of a variant of chain compactness together with omitting types, which has been extended by the first author to other cardinals \cite{b-mtlc}. We will need to incorporate omitting types into our pseudo-models framework in order to give a reformulation for virtually supercompact cardinals.
\begin{defin}
We will say that a $\delta$-pseudo-model $\cM$ in a language $\tau^*$ \emph{omits} an $\cL(\tau)$-type $p(x)$ if there is a $\delta$-forth system $\cF$ from $\tau$ to $\tau^*$ such that, for all $f:\sigma\to\sigma^*$ from  $\cF$, $\cM$ models $f^*\image (T\cap \cL(\sigma))$ and omits $f^*\image (p\cap \cL(\sigma))$.
\end{defin}
\begin{defin}
We will say that an $\cL(\tau)$-theory $T$ is \emph{increasingly filtered by $\cP_\kappa\delta$} if $T$ is the union of a sequence $\vec T=\langle T_s\mid s\in \cP_\kappa\delta\rangle$ such that whenever $s\subseteq s'$, then $T_s\subseteq T_{s'}$, and we will call $\vec T$, an \emph{increasing filtration} of $T$.

\end{defin}

\begin{theorem}\label{th:virtuallySupercompactCharacterization} The following are equivalent for a cardinal $\kappa$.
\begin{enumerate}
\item $\kappa$ is virtually supercompact $($remarkable$)$.
\item For every $\delta>\kappa$, whenever $T$ is an $\bL_{\kappa,\kappa}(\tau)$ theory that is increasingly filtered by $\cP_{\kappa}\delta$ and $p^a(x)$ for $a\in A$ is some set of types each of which is increasingly filtered by $\cP_{\kappa}\delta$ such that every $T_s$ has a model omitting all $p^a_s(x)$, then there is a pseudo-model of $T$ omitting all $p^a(x)$.

\item Same as $(2)$ but with pseudo-model replaced by $\kappa^+$-pseudo-model.
\end{enumerate}
\end{theorem}
\begin{proof}
Suppose that $\kappa$ is virtually supercompact. Fix an $\bL_{\kappa,\kappa}(\tau)$-theory $T=\bigcup_{s\in\cP_{\kappa}\delta} T_s$ with an increasing filtration $\vec T=\langle T_s\mid s\in\cP_\kappa\delta\rangle$ and $\bL_{\kappa,\kappa}(\tau)$-types $p^a(x)=\bigcup_{s\in\cP_{\kappa}\delta} p_s^a(x)$ indexed by $a\in A$ with increasing filtrations $\vec p^{\,a}=\langle p^a_s(x)\mid s\in \cP_\kappa\delta\rangle$ satisfying the hypothesis of the theorem. Let $\lambda$ be a large enough $\beth$-fixed point of cofinality $\kappa^+$ so that $V_\lambda$ sees all this. Choose $\alpha>\lambda$ such that there is a transitive model $\cN$ closed under $\lambda$-sequences and a virtual elementary embedding $j:V_\alpha \to \cN$ with $\crit j=\kappa$ and $j(\kappa)>\lambda$. Consider the restriction $j:V_\lambda\to j(V_\lambda)$. Observe that since $V_\lambda$ is closed under $\kappa$-sequences by cofinality considerations, $j(V_\lambda)$ is closed under $j(\kappa)$-sequences in $\cN$ by elementarity. Thus, $j(V_\lambda)$ is truly closed under $\lambda$-sequences, since $\cN$ is. We will not use the embedding $j$, but move to a ${\rm Coll}(\omega,V_\lambda)$-extension $V[G]$ with an elementary embedding $h:V_\lambda\to \bar \cN =j(V_\lambda)$ with $\crit h=\kappa$ and $h(\kappa)>\lambda$.

Since the forcing ${\rm Coll(\omega,V_\lambda)}$ has size $|V_\lambda|=\lambda$ because we chose $\lambda$ to be a $\beth$-fixed point, we can cover $h\image\lambda$ by a set $s^* \subseteq j(\lambda)$ of size $\lambda<h(\kappa)$ in $V$. Next, observe crucially that $s^*$ is an element of $\bar\cN$ by $\lambda$-closure. By elementarity, $\bar\cN$ satisfies that $h(\vec T)_{s^*}$ has a model $\cM$ omitting all $h(p)^a_{s^*}$ for $a\in j(A)$. Since $h\image\delta \subseteq s^*$, we have for every $s\in\cP_\kappa\delta$ that $h(s)=h\image s\subseteq s^*$. It follows that $h\image T_s\subseteq h(\vec T_s)=h(\vec T)_{h(s)}\subseteq h(\vec T)_{s^*}$, and similarly $h\image \vec p^{\,a}_s\subseteq h(\vec p^{\,a}_s)\subseteq h(\vec p)_{s^*}^{\,h(a)}$. It follows that $\cM$ is the required $\kappa^+$-pseudo-model.

In the other direction, suppose that we have the compactness assumption in (2) and fix a singular $\beth$-fixed point $\lambda>\kappa$ and $\alpha>\lambda$. Let $\tau$ be the language consisting of a binary relation $\in$ and constants $\{ c_x\mid x\in V_\alpha\}\cup \{d_x \mid x\in V_\lambda\}\cup \{c\}$. Let $T$ be the following $\bL_{\kappa,\kappa}(\tau)$-theory:
$${\rm ED}_{\bL_{\kappa,\kappa}}(V_\alpha,\in,c_x)_{x\in V_\alpha}\cup \{c_\xi\neq c<c_\kappa\mid \xi<\kappa\}\cup \{d_b\in d_a\mid b\in a,\,a\in V_\lambda\}\cup\{d_\xi<d_\eta<c_\kappa\mid \xi<\eta<\lambda\},$$ and let $p^a(x)$, for $a\in V_\lambda$, be the following $\bL_{\kappa,\kappa}(\tau)$-types:
$$\{x\in d_a\}\cup \{x\neq d_b\mid b\in a\}.$$
Now let us find a filtration for the theory $T$ and the types $p^a(x)$ such that $V_\alpha$ can be made, by correctly interpreting the constants $d_a$, into a model of $T_s$ omitting all $p^a_s(x)$. Fix a bijection $f:\lambda\to V_\lambda$. Given $s\in \cP_\kappa \lambda$, let $X_{s}\prec V_\alpha$ be the elementary substructure of $V_\alpha$ generated by $f\image s\subseteq V_\lambda$. Let ${\rm ED}_{\bL_{\kappa,\kappa}}(V_\alpha,\in,c_x)_{x\in V_\alpha}\subseteq T_s$, but $T_s$ is only allowed to mention sentences $c_\xi\neq c<c_\kappa$ if $c_\xi\in X_s$, it is only allowed to mention sentences with constants $d_a$ if $a\in X_s$. Let $p^a_s(x)=\emptyset$ if $a\notin X_s$. Otherwise, suppose $a\in X_s$. In this case, let $p^a_s(x)$ mention only formulas $x\neq d_b$ for $b\in X_s$. Let $\pi:X_s \to M$ be the Mostowski collapse. To make $V_\alpha$ into a model of $T_s$ omitting $p^a_s(x)$, we will first interpret all $c_x$ as $x$. For every $b\in X_s\cap V_\lambda$, let $d_b$ be interpreted as $\pi(b)$. This ensures that there is no space to interpret $x$ to satisfy $p^a_s(x)$.

By assumption, $T$ has a well-founded pseudo-model $\cM$ omitting all $p^a(x)$, and we can assume without loss that it is transitive. Let $\cF$ be the associated forth system between the languages $\tau$ and $\tau^*$, and let us force with $\cF$ to add a generic injection $F:\tau\to\tau^*$. Define $F_*(\phi(x))=f_*(\phi(x))$ for some/any $f\in \cF$ such that $f\subseteq F$. In the forcing extension, we get an elementary embedding $j:V_\alpha\to \cM$ with $\crit j=\kappa$, $j(\kappa)>\lambda$ (we cannot have $j(\kappa)=\lambda$ since we chose $\lambda$ to be singular) and the model $\cM$ omits all types $p^{*a}(x):=F_*\image p^a(x)$. It follows that $\cM$ has a transitive subset isomorphic to $V_\lambda$ and so $V_\lambda\subseteq \cM$.
\end{proof}
The proof of Theorem~\ref{th:virtuallySupercompactCharacterization} shows that a virtual version of strongly compact cardinals is equivalent to virtual supercompactness.
\begin{defin}
A cardinal $\kappa$ is \emph{virtually strongly compact} if and only if for every $\lambda>\kappa$, there is $\alpha>\kappa$ and a transitive $\cM$ with $\cM^{{<}\kappa}\subseteq \cM$ and $s\in \cM$ such that there is a virtual elementary embedding $j:V_\alpha\to \cM$ with $\crit j=\kappa$, $j(\kappa)>\lambda$, $|s|^{\cM}<j(\kappa)$ and $j\image\lambda\subseteq s$.
\end{defin}
\begin{theorem}
A cardinal $\kappa$ is virtually strongly compact if and only if it is virtually supercompact.
\end{theorem}
\begin{proof}
It suffices to observe, from the proof of Theorem~\ref{th:virtuallySupercompactCharacterization}, that the compactness property of Theorem~\ref{th:virtuallySupercompactCharacterization} follows from virtual strong compactness. Note that we need the closure on $\cM$ to verify that it is correct about models of $\bL_{\kappa,\kappa}$.
\end{proof}
It is not difficult to see that virtually strongly compact cardinals are the $\kappa^+$-pseudo-compactness cardinals for $\bL_{\kappa,\kappa}$. Thus, virtually supercompact cardinals are precisely the $\kappa^+$-pseudo-compactness cardinals for $\bL_{\kappa,\kappa}$.
\begin{theorem}
The following are equivalent.
\begin{enumerate}
\item $\kappa$ is virtually strongly compact.
\item $\kappa$ is a $\kappa^+$-pseudo-compactness cardinal for $\bL_{\kappa,\kappa}$. 
\item $\kappa$ is an $\omega$-pseudo-compactness cardinal for $\bL_{\kappa,\kappa}$.
\end{enumerate}
\end{theorem}
It is, of course, the case that $\kappa$ is strongly compact if and only if it is a strong compactness cardinal for $\bL_{\kappa,\omega}$. 
\begin{question}
If $\kappa$ is a $\kappa^+$-pseudo-compactness cardinal for $\bL_{\kappa,\omega}$, is it virtually strongly compact? 
\end{question}
Recall that \Vopenka's principle is the assertion that every proper class of first-order structures in the same language has two structures which elementarily embed. Analogously, \emph{virtual \Vopenka's principle} (also known in the literature as \emph{generic \Vopenka's principle}) asserts that every proper class of first-order structures in the same language has two structures which virtually elementarily embed.

Makowsky showed that Vop\v{e}nka's principle is equivalent to the assertion that every logic has a strong compactness cardinal \cite{m-vopcomp}. Bagaria showed that Vop\v{e}nka's principle is equivalent to the assertion that for every $n<\omega$, there is a $C^{(n)}$-extendible cardinal \cite{b-cndard}. The third author and Hamkins showed in  \cite{gh-genvop} that virtual \Vopenka's principle is equivalent to the assertion that for every $n<\omega$, there is a proper class of weakly virtually $C^{(n)}$-extendible cardinals, but at the same same time it is consistent that virtual \Vopenka's principle holds and yet there are no even virtually supercompact cardinals.

We will reprove Moskowsky's theorem and show that one direction generalizes to the case of virtually $C^{(n)}$-extendible cardinals, but the other direction fails to generalize because of the split in the virtual case into the weak and strong forms of $C^{(n)}$-extendibility.
\begin{theorem}[Makowsky, Bagaria]\label{th:makowsky}
For every $n<\omega$, there is a $C^{(n)}$-extendible cardinal if and only if every logic has a strong compactness cardinal.
\end{theorem}
We will need the following easy fact about $C^{(n)}$-extendible cardinals.
\begin{prop}[\cite{b-cndard}] Suppose that for every $n<\omega$, there is a $C^{(n)}$-extendible cardinal. Then for every $n<\omega$, there is a proper class of $C^{(n)}$-extendible cardinals.
\end{prop}
\begin{proof}
Suppose towards a contradiction that for some fixed $n$, the $C^{(n)}$-extendible cardinals are bounded by $\delta$. Let $\kappa$ be a $C^{(m)}$-extendible cardinal for some $m\gg n$. Obviously $\kappa<\delta$. Fix $\alpha>\delta$ in $C^{(m)}$ and take an elementary embedding $j:V_\alpha\to V_\beta$ with $\crit j=\kappa$, $j(\kappa)>\alpha$ and $\beta\in C^{(m)}$. By elementarity, $V_\alpha$ sees that $\kappa$ is $C^{(n)}$-extendible. It follows that $V_\beta$ thinks that $j(\kappa)$ is $C^{(n)}$-extendible and it must be correct about this by the level of elemenarity. But $j(\kappa)>\alpha>\delta$, which is the desired contradiction.
\end{proof}
\noindent Clearly the argument would work identically for virtually $C^{(n)}$-extendible cardinals as well, but not for weakly virtually $C^{(n)}$-extendible cardinals.
\begin{proof}[Proof of Theorem~\ref{th:makowsky}]
Fix a logic $\cL$ with occurrence number $o(\cL)=\delta$. The logic $\cL$ and satisfaction relation $\vDash_{\cL}$ are defined by some $\Sigma_n$-formulas with a parameter $a$ of rank $\delta_a$. Clearly there are only set-many languages of size less than $\delta_a$ modulo a renaming. So let us fix an ordinal $\delta>\delta_a$ such that $V_{\delta}$ is closed under $\cL$ and any language $\tau$ of size less than $\delta_a$ has a renaming to a language $\tau'\in V_{\delta}$. Let $\kappa>\delta$ be $C^{(n)}$-extendible. We will argue that $\kappa$ is a strong compactness cardinal for $\cL$. Let $T$ be a ${<}\kappa$-satisfiable $\cL(\tau)$-theory. Choose an elementary embedding $j:V_\alpha\to V_\beta$ with $\alpha,\beta\in C^{(n)}$, $\crit j=\kappa$ and $j(\kappa)>\alpha$ such that $V_\alpha$ sees that $T$ is ${<}\kappa$-satisfiable. Note that both $V_\alpha$ and $V_\beta$ are correct about $\cL$ and $\vDash_{\cL}$ because they are $\Sigma_n$-elementary in $V$.

By elementarity $V_\beta$ satisfies that $j\image T$ has a model $M$ in the logic $\cL(j\image\tau)$, and since $V_\beta\prec_{\Sigma_n}V$, it must be correct about this. Let $f:\tau\to j\image\tau$ be the renaming taking elements of $\tau$ to their images under $j$. Under the renaming $M$ is a $\tau$-structure, and it suffices to show that $f_*(\phi)=j(\phi)$ for every $\phi\in T$. Fix $\phi\in T$ and let $\tau_\phi$ be the ${<}\delta$-sized subset of $\tau$ used in $\phi$. Let $f^\phi=f\restriction\tau_\phi$. By our assumptions, there is a language $\sigma\in V_\kappa$ with $\cL(\sigma)\in V_\kappa$ and a renaming $g:\sigma\to \tau_\phi$. Let $g_*(\bar\phi)=\phi$. By elementarity, $j(g):\sigma\to j(\tau_\phi)$ is a renaming from $\sigma$ to $j(\tau_\phi)=j\image\tau_\phi$. It is easy to check that $j(g)=f^\phi\circ g$. Now, by elementarity, we have $(j(g))_*(\bar\phi)=j(g_*)(\bar\phi)=j(g_*(\bar\phi))=j(\phi)$. By our observations from Section~\ref{sec:prelim} we have, $j(g)_*=(f^\phi\circ g)_*=f^\phi_*\circ g_*$. Thus, we have $f^\phi_*(\phi)=f^\phi_*(g_*(\bar \phi))=j(g)_*(\bar\phi)=j(\phi)$. Again, by our observations from Section~\ref{sec:prelim} about restrictions of renamings, we have $f_*(\phi)=j(\phi)$, but in fact it already suffices to know that $f^\phi_*(\phi)=j(\phi)$ since $V_\beta$ is correct about satisfaction for $\cM$.


In the other direction, we will argue that if there is a strong compactness cardinal for the sort logic $\bL^{s,\Sigma_n}$, then there  must be a $C^{(n)}$-extendible cardinal. So suppose that $\gamma$ is a strong compactness cardinal for $\bL^{s,\Sigma_n}$. Fix $\alpha>\gamma$ in $C^{(n)}$. We can write the usual theory whose model gives an elementary embedding $j:V_\alpha\to V_\beta$ with $\crit j=\kappa_\alpha\leq\gamma$, and using sort logic, by Proposition~\ref{prop:sort}, we can express that $\beta\in C^{(n)}$. Since there are only boundedly many $\kappa\leq\gamma$, by the pigeon-hole principle we can choose some $\kappa_{\alpha^*}$ which works for unboundedly many $\alpha$, and this $\kappa=\kappa_{\alpha^*}$ must be $C^{(n)}$-extendible. Note that because we are not in the virtual case we do not need to show additionally that $j(\kappa)>\alpha$.
\end{proof}

The proof above gives us the following results for the virtual case.
\begin{theorem} If for every $n<\omega$, there is a virtually $C^{(n)}$-extendible cardinal, then every logic has a $\kappa^+$-pseudo-compactness cardinal $\kappa$.
\end{theorem}
\begin{proof}
Following the proof of Theorem~\ref{th:makowsky} for the forward direction, it suffices to observe that even though $j\image T$ may not be in $V_\beta$, we can cover it by a theory $T^*\in V_\beta$ of size less than $j(\kappa)$.
\end{proof}

The first author showed in \cite{b-mtlc} that a cardinal $\kappa$ is $C^{(n)}$-extendible if and only if $\kappa$ is a strong compactness cardinal for $\bL^{s,\Sigma_n}_{\kappa,\omega}$. This result can be reformulated in the pseudo-compactness framework.
\begin{theorem}
A cardinal $\kappa$ is virtually $C^{(n)}$-extendible if and only if $\kappa$ is a $\kappa^+$-pseudo-compactness cardinal for $\bL^{s,\Sigma_n}_{\kappa,\omega}$.
\end{theorem}
\begin{proof}
The forward direction follows from the proof of Theorem~\ref{th:makowsky} and the backward direction follows because in the infinitary logic $\bL_{\kappa,\omega}$ we can express that the critical point of our embedding is exactly $\kappa$.
\end{proof}

\noindent Indeed, as in previous arguments, we do not need $\kappa^+$-pseudo models in the argument; having pseudo-models suffices to obtain the desired virtual elementary embeddings.
\begin{cor} The following are equivalent.
\begin{enumerate}
\item For every $n<\omega$, there is a virtually $C^{(n)}$-extendible cardinal.
\item Every logic has a $\kappa^+$-pseudo-compactness cardinal.
\item Every logic has an $\omega$-pseudo-compactness cardinal.
\end{enumerate}
\end{cor}


\begin{cor}
Virtual \Vopenka's principle is not equivalent to the assertion that every logic has a  $\kappa^+$-pseudo-compactness cardinal $\kappa$.
\end{cor}
We do, however, get a characterization of virtual Vopenka's principle in terms of chain compactness.
\begin{theorem} The following are equivalent.
\begin{enumerate}
\item Virtual \Vopenka's principle.
\item Every logic has a $\kappa^+$-pseudo-chain compactness cardinal $\kappa$.
\item Every logic has an $\omega$-pseudo-chain compactness cardinal $\kappa$.
\end{enumerate}
\end{theorem}
\begin{proof}
Suppose that virtual \Vopenka's principle holds and so for every $n<\omega$, we have a proper class of weakly virtually $C^{(n)}$-extendible cardinals. Then (3) then follows directly by the proof of Theorem~\ref{th:makowsky}.

Assume (2). We need to show that given a fixed ordinal $\beta$, there is $\gamma>\beta$ which is weakly virtually $C^{(n)}$-extendible. We use the logic $\bL^{s,\Sigma_n}_{\delta,\omega}$ for some $\delta>\beta$ to ensure that the critical point of the virtual embedding we obtain is above $\beta$ and argue as in the proof of Theorem~\ref{th:makowsky}.
\end{proof}
\begin{cor}\label{cor:WeakCompactnessVopenka}
If virtual \Vopenka's principle holds, then every logic has a weak compactness cardinal.
\end{cor}

The above analysis was carried out in ${\rm ZFC}$ permitting only the definable classes. However, all of it straightforwardly generalizes to the second-order context of ${\rm GBC}$ by replacing the notion of (virtually) $C^{(n)}$-extendible with the notion of (virtually) $A$-extendible cardinals for a class $A$. A cardinal $\kappa$ is \emph{$A$-extendible}, for a class $A$, if for every $\alpha>\kappa$, there is $\beta>\kappa$ such that there is an elementary embedding $j:(V_\alpha,\in,A\cap V_\alpha)\to (V_\beta,\in,A\cap V_\beta)$ with $\crit j=\kappa$ and $j(\kappa)>\alpha$. The virtual versions are defined analogously (for more details, see \cite{gh-genvop}). Thus, we get for example, that in ${\rm GBC}$, \Vopenka's principle holds if and only if for every class $A$, there is an $A$-extendible cardinal, etc.

In Definition \ref{forth-def}, we introduced the notion of a forth system between languages.  Below, we define a forth system between models of the same language (Definition \ref{game-def}.(\ref{forth-model})).  First, we explain the motivation and their connection to back-and-forth systems and Ehrenfeucht-\Fraisse\ games.

It is a classical result that two structures $M$ and $N$ from $V$ are isomorphic in a forcing extension if and only if the \emph{good} player has a winning strategy in the $\omega$-length Ehrenfeucht-\Fraisse\ game $\G_F(M,N)$. The game starts with the \emph{bad} player choosing to play an element $a_0\in M$ or $b_0\in N$ and the \emph{good} player has to respond with a move $b_0\in  N$ or $a_0\in M$ respectively so that the map $f=\{\langle a_0,b_0\rangle\}$ is a finite partial isomorphism between $M$ and $N$. Given a finite partial isomorphism $f$ between $M$ and $N$ that results from an initial play of $\G_F(M,N)$, at the next step of the game, the bad player again chooses to play an element out of either $M$ or $N$ and the good player has to respond to extend $f$ to a partial isomoprhism. It is easy to see that the good player having a winning strategy in $\G_F(M,N)$ is equivalent to the existence of a back-and-forth system. A \emph{back-and-forth system} between two models $M$ and $N$ in the same language $\tau$ is a collection $\cP$ of finite partial isomorphisms between $M$ and $N$ such that whenever $a\in M$ and $f\in\cP$, then there is $g\in \cP$ extending $f$ with $a\in\dom g$, and conversely whenever $b\in N$, then there is a $g'\in \cP$ extending $f$ with $b\in\text{ran}(g')$. Clearly, forcing with a back-and-forth system gives a virtual isomorphism and, in the other direction, a virtual isomorphism suffices to obtain a back-and-forth system via the forcing relation. The existence of a back-and-forth system is also equivalent to the assertion that the models $M$ and $N$ are elementary equivalent in the  quasi-logic $\bL_{{\rm Ord},\omega}$.

Schindler showed that two structures from $V$ have a virtual embedding if and only if the \emph{good} player has a winning strategy in the $\omega$-length modified Ehrenfeucht-\Fraisse\ game $\G_v(M,N)$ \cite{bgs-virt-vop}. The game starts with the \emph{bad} player playing an element $a_0\in M$ and the \emph{good} player has to respond with a move $b_0\in N$ so that the map $f=\{\langle a_0,b_0\rangle\}$ is a finite partial isomorphism between $M$ and $N$. The good player wins if at each step she can maintain a finite partial isomorphism between $M$ and $N$. For the game $\G_v(M,N)$, the corresponding notion to a back-and-forth system is a \emph{forth-system} $\cP$, which is a collection of finite partial isomorphisms $f$ between $M$ and $N$ with only the forth extension property. The corresponding quasi-logic will be a subclass of $\bL_{{\rm Ord},\omega}$, which we will call \emph{virtual logic}. Indeed, we can define the notion of a virtualization for any abstract logic.

\begin{defin}\label{game-def}
Fix a logic $\cL$.
\begin{enumerate}
 \item Given two structures $M$ and $N$ in the same language, the game $\G_{v_\cL}(M, N)$ is defined exactly as the game $\G_{v}(M,N)$ with  satisfaction given by the logic $\cL$ in place of first-order logic.
    \item The quasi-logic $\cL^v(\tau)$, for a language $\tau$, consists of the formulas given by the following closure rules.
    \begin{enumerate}
        \item Every formula of $\cL(\tau)$ is a formula of $\cL^v(\tau)$,
        \item If $\phi(x) \in \cL^v(\tau)$ (with possibly other free variables), then $\exists x \phi(x) \in \cL^v(\tau)$.
        \item If $\{\phi_i \mid i \in I\}$ is a collection of formulas from $\cL^v(\tau)$ jointly in finitely many variables, then so is $\bigwedge_{i\in I} \phi_i$.
    \end{enumerate}

    \item \label{forth-model}An $\cL$-forth system $\cP$ from a $\tau$-structure $M$ to a $\tau$-structure $N$ is a collection of $\cL$-elementary embeddings with the following properties.
    \begin{enumerate}
        \item $\emptyset \in \cP$.
        \item If $f \in \cP$ and $a \in M$, then there is $g \supseteq f$ in $\cP$ such that $a \in \dom g$.
    \end{enumerate}
\end{enumerate}
\end{defin}
Recall, that technically our abstract logic does not support formulas with free variables; instead something like ``$\phi(x) \in \cL^v(\tau)$" with $x$ free really means ``$\phi(c) \in \cL^v(\tau\cup\{c\})$", where $c$ is a new constant symbol.  However, we use the notation of free variables for better readability.

\begin{theorem}\label{th-virt}
For structures $M$ and $N$ in the same language $\tau$, the following are equivalent.
\begin{enumerate}
 \item The good player has a winning strategy in $\G_{\cL_v}(M, N)$.
    \item $N \vDash Th_{\cL^v}(M)$, the collection of all sentences in $\cL^v(\tau)$ that $M$ satisfies.
    \item There is an $\cL$-forth system from $M$ to $N$.
    \item There is a virtual $\cL$-elementary embedding $f:M\to N$.

\end{enumerate}
\end{theorem}

The new parts of Theorem~\ref{th-virt} are proved by varying classical proofs involving Ehrenfeucht-\Fraisse\ games. If $M$ and $N$ satisfy any one of the equivalent conditions of Theorem~\ref{th-virt}, then we shall say that $M$ is a \emph{virtual submodel} of $N$. This lets us give L\"{o}wenheim-Skolem-Tarski style characterizations of virtual large cardinals. For example, a virtualization of classical arguments about supercompact cardinals from \cite{m-roleof} show:
\begin{theorem}
Suppose $\kappa$ is virtually supercompact and $\Psi$ is an $\bL^2$-sentence. If $M$ is an $\tau(\Psi)$-structure such that $M \vDash \Psi$, then $M$ has a virtual submodel of size less than $\kappa$ that also models $\Psi$.
\end{theorem}
\begin{proof}
Magidor showed that a cardinal $\kappa$ is supercompact if and only if for every $\delta>\kappa$, there is $\bar\delta<\kappa$ and an elementary embedding $j:V_{\bar\delta}\to V_\delta$ with $\crit j=\gamma$ and $j(\gamma)=\kappa$. The obvious virtualized reformulation holds for virtually supercompact cardinals as well, and this immediately implies the desired reflection.
\end{proof}
\noindent Magidor also showed that the least cardinal with this reflection property is supercompact. We do not expect the result to generalize to the virtual context for the usual reasons.
\begin{question}
Is the least cardinal $\kappa$ with the property that for every $\bL^2$-sentence $\Psi$, every $\tau(\Psi)$-structure $M\vDash\Psi$ has a virtual submodel of size less than $\kappa$ that also satisfies $\Psi$ virtually supercompact?
\end{question}

Finally, using virtual logic we can give an alternative compactness-style characterization of weakly virtually extendible cardinals asserting that whenever $T \subseteq \bL_{\kappa,\kappa}^2(\tau^*)$ has size $\kappa$ and all of its ${<}\kappa$-sized pieces can be realized in expansions of a single model $\cM$ of some sub-language $\tau\subseteq \tau^*$, then $T$ has a model $\cN$ also satisfying the virtual theory of $\cM$.

\begin{theorem}
A cardinal $\kappa$ is weakly virtually extendible if and only if for every theory $T \subseteq \bL_{\kappa,\kappa}^2(\tau^*)$ of size $\kappa$, if $\cM$ is a $\tau$-structure, with $\tau\subseteq\tau^*$, such that for every $T^*\subseteq T$ of size less than $\kappa$, there is an expansion of $\cM$ to a $\tau^*$-structure making $\cM\vDash T^*$, then there is a $\tau^*$-structure $\cN$ satisfying  $Th_{(\bL^2_{\kappa,\omega})^v(\tau)}(\cM)$ together with $T$.
\end{theorem}
\begin{proof} Suppose the compactness characterization holds and fix $\alpha > \kappa$. Let $\tau^*$ be the language consisting of a binary relation $\in$ and constants $\{c_\xi\mid\xi\leq\kappa\}\cup\{c\}$ and let $\tau$ be $\tau^*$ without $c$. Let $T$ be the following $\bL_{\kappa,\omega}^2(\tau)$-theory.
$${\rm ED}_{\bL_{\kappa,\omega}}(V_\alpha,\in,c_\xi)_{\xi\leq\kappa}\cup \{c_\xi<c<c_\kappa\mid \xi<\kappa\}\cup \{\Phi\}.$$
It is easy to see that if $T^*$ is any sub-theory of $T$ of size less than $\kappa$, then we can interpret $c$ by a large enough $\lambda<\kappa$. So by hypothesis, there is a virtual $\bL_{\kappa,\omega}(\tau)$-embedding \hbox{$j:(V_\alpha,\in,c_\xi)_{\xi\leq\kappa}\to (V_\beta,\in,c_\xi)_{\xi\leq\kappa}$}. The only question is how each $c_\xi$ is interpreted in $V_\beta$. Since $V_\beta$ satisfies the $\bL_{\kappa,\omega}$-theory of $(V_\alpha,\in)$, it knows that for $\xi<\kappa$, that $c_\xi$ must be interpreted as $\xi$. Also, since $V_\beta$ is a model $T$, it knows that $c_\kappa>\kappa$. Therefore $\kappa$ is the critical point of $j$.

In the other direction, suppose that $\kappa$ is weakly virtually extendible. Fix a $\kappa$-sized $\bL_{\kappa,\kappa}^2(\tau^*)$-theory $T$ and let $\cM$ be a $\tau$-structure as in the hypothesis. Since $\tau^*$ has size $\kappa$, we can assume that both $\tau^*$ and $T$ are subsets of $V_\kappa$. Fix $V_\alpha$ large enough that it contains the model $\cM$ and fix a virtual elementary embedding $j:V_\alpha\to V_\beta$ with critical point $\kappa$. So $V_\beta$ satisfies that for every $T^*\subseteq j(T)$ of size less than $j(\kappa)$ there is an expansion of $j(\cM)$ to a $j(\tau^*)$-structure such that $j(\cM)\vDash T^*$. In particular, there is an expansion of $j(\cM)$ to a $\tau^*$-structure satisfying $j\image T=T$. It is not difficult to see that $j:\cM\to j(\cM)$ is a virtual $\bL^2_{\kappa,\kappa}$-elementary embedding. Thus, $j(\cM)$ is a model of $T$ that satisfies the virtual theory $Th_{(\bL^2_{\kappa,\omega})^v(\tau)}(\cM)$ .
\end{proof}

\section{\Vopenka\ for weak compactness}\label{sec:vopenka-weak}

Recall that Vop\v{e}nka's principle is equivalent to the assertion that every logic has a strong compactness cardinal (Theorem \ref{th:makowsky}).  In this section, we explore the assertion that every logic has a \emph{weak} compactness cardinal.  Corollary \ref{cor:WeakCompactnessVopenka} shows that this follows from the virtual version of Vop\v{e}nka's Principle and, in particular, is consistent with $L$.  The main result of this section (Theorem \ref{th:subtle}) connects this weak compactness principle with \emph{subtlety}, although with two important caveats:
\begin{itemize}
	\item we assume that there are \emph{stationarily many} weak compactness cardinals; and
	\item we use global choice (and the formalism of G\"{o}del-Bernays set theory).
\end{itemize}
We show that each of these points is necessary for the connection with subtlety (see Theorems \ref{thm-not-stat} and \ref{th:globalChoiceCE}).

This section deals with the set-theoretic principle \emph{${\rm Ord}$ is subtle} (Definition \ref{def-subtle}), which is an assertion about properties of classes\footnote{All of these arguments could be rephrased using the principle `$\kappa$ is subtle,' but we prefer the `${\rm Ord}$ is subtle' formalism due to its similarity to Vop\v{e}ka's Principle (versus $\kappa$ is Vop\v{e}nka).  Additionally, this formalism helps us see the subtleties involved with global choice.}. So we will start by explaining some frameworks that we can use for our formal setting. In first-order set theory, classes are definable with (parameters) collections of sets, and therefore objects of the meta-theory. In second-order set theory, classes are actual elements of the model. Second-order set theories are formalized in a two-sorted logic with objects for both sets and classes. A standard axiomatization of second-order set theory is G\"odel-Bernays set theory ${\rm GBC}$. The theory ${\rm GBC}$ consists of ${\rm ZFC}$ together with the following axioms for classes: extensionality, replacement (a class function restricted to a set is a set), global choice axiom (there is a well-ordering of sets), and the comprehension scheme for first-order formulas asserting that every first-order formula (with set/class parameters) defines a class. A first-order model of ${\rm ZFC}$ together with its definable collections satisfies all the axioms of ${\rm GBC}$ except possibly global choice. Theorem~\ref{th:globalChoiceCE} below gives a standard class forcing construction of a ${\rm ZFC}$-universe without a definable global well-order. Since our results below will require global choice, we will have to work either in ${\rm GBC}$ or over a first-order ${\rm ZFC}$-model with a definable global well-order, such as the constructible universe $L$. Alternatively, we can assume that we are working with a rank-initial segment $V_\kappa$, for an inaccessible cardinal $\kappa$ and our classes are precisely $V_{\kappa+1}$. The second-order model consisting of the sets $V_\kappa$ and classes $V_{\kappa+1}$ satisfies ${\rm GBC}$ with full comprehension for all second-order assertions, which together comprise the much stronger second-order set theory Kelley-Morse.

Note that if $\kappa$ is a strong compactness cardinal for a logic $\cL$, then every $\lambda>\kappa$ is also a strong compactness cardinal for $\cL$. This does not hold for weak compactness cardinals, however we do have the following observation.

\begin{obs}
If every logic $\cL$ has a weak compactness cardinal, then every logic $\cL$ has unboundedly many weak compactness cardinals.
\end{obs}

\begin{proof}
If $\kappa$ is a weak compactness cardinal for $\cL$, then we can find a larger weak compactness cardinal for $\cL$ by finding a weak compactness cardinal for $\cL \cup \bL_{\kappa^+,\omega}$.
\end{proof}

We start the proof of the main theorem of this section with the following proposition asserting that the ordinals $\alpha<\beta$ given by subtleness can be assumed to be regular cardinals.

\begin{prop}\label{prop:subtle}
Assume ${\rm GBC}$. If ${\rm Ord}$ is subtle and $C$ is a class club, then for any class sequence $\seq{A_\xi\mid \xi\in{\rm Ord} }$ with $A_\xi\subseteq \xi$, there are regular cardinals $\alpha < \beta$ in $C$ such that $A_\alpha = A_\beta \cap \alpha$.
\end{prop}
The proof is an easy modification of the proof of \cite[Theorem 3.6.3]{ahkz-flip}.
\begin{proof}
 By shrinking $C$ if necessary, we can suppose that $C$ consists only of cardinals.  First, we define an auxiliary sequence $\seq{B_\alpha\mid  \alpha\in C}$ by:
\begin{enumerate}
	\item If $\alpha$ is singular with cofinality $\mu$, then we fix a cofinal sequence $s^{(\alpha)}$ of length $\mu$ in $\alpha$ and define
	$$B_\alpha= \{(0,\mu), (1, s^{(\alpha)}_\xi) \mid \xi < \mu\}.$$
	\item If $\alpha$ is regular, then
	$$B_\alpha = \{(2, \xi) \mid \xi \in A_\alpha\}.$$
\end{enumerate}
Note that we make use here of global choice to pick the cofinal sequences $s^{(\alpha)}$. By applying subtlety to the club $C$ and the sequence $\seq{B_\alpha \mid \alpha \in C}$, there are $\alpha < \beta$ from $C$ such that $B_\alpha = B_\beta \cap \alpha$.  If $\beta$ is regular, then $\alpha$ must be regular because $B_\alpha$ consists of pairs of the form $(2, \xi)$.  So suppose $\beta$ is singular. Then $\alpha$ must also be singular and $(0, \cf\beta) \in B_\alpha$.  So $\cf \alpha = \cf \beta$.  It follows that $\seq{ \nu \mid (1, \nu) \in B_\alpha}$ is a $\cf \beta$-sequence  cofinal in $\alpha$ that is end-extended to a $\cf \beta$-sequence cofinal in $\beta$, but this is impossible.  So $\alpha$ and $\beta$ are both regular and also, $A_\alpha = A_\beta \cap \alpha$ by the definition of the $B_\alpha$-sequence.
\end{proof}

\begin{theorem}\label{th:subtle}
The following are equivalent over ${\rm GBC}$.
\begin{enumerate}
	\item ${\rm Ord}$ is subtle.
	\item Every logic has a stationary class of weak compactness cardinals.
\end{enumerate}
\end{theorem}

\begin{proof}
First, suppose that ${\rm Ord}$ is subtle and fix a logic $\cL$ with occurrence number $o(\cL)=\lambda$. We will assume towards a contradiction that there is a class club $C_\cL$ of cardinals such that no element of $C_{\cL}$ is weakly compact for $\cL$. Using global choice, we can choose, for every $\alpha \in C_\cL$, a language $\tau_\alpha$ and an $\cL(\tau_\alpha)$-theory $T_\alpha$ of size $\alpha$ that is ${<}\alpha$-satisfiable, but not satisfiable.

We can assume that the first element of $C$ is very high above $\lambda$, so that each $\tau_\alpha$ has size at most $\alpha$. But then we can further assume that each $\tau_\alpha$ is the `maximal' language of size $\alpha$, having relations $R^n_\xi$ of arity $n$, functions $F^n_\xi$ of arity $n$, and constants $c_\xi$, for $n<\omega$ and $\xi<\alpha$. Thus, in particular, we have that $\tau_\alpha$ extends $\tau_\beta$ for $\beta<\alpha$, and that the sequence of the $\tau_\alpha$ is continuous. This, in turn, gives that the sequence of the $\cL(\tau_\alpha)$ is increasing and continuous at cardinals of cofinality $\geq\lambda$. Let's code the sentences of $\cL$ by ordinals and from now on associate elements of $\cL$ with their ordinal codes. Now observe that cardinals $\alpha$ such that $\cL(\tau_\alpha)\subseteq\alpha$ form an unbounded class $D$ that is closed under sequences of cofinality $\geq\lambda$. It is easy to find $\alpha$ such that $\bigcup_{\xi<\alpha}\cL(\tau_\xi)\subseteq\alpha$ and if $\cf \alpha\geq\lambda$, then, using $o(\cL)=\lambda$ and the continuity of the $\tau_\alpha$ sequence, $\bigcup_{\xi<\alpha}\cL(\tau_\xi)=\cL(\tau_\alpha)$. Let $\bar D$ be $D$ together with all its limit points of small cofinality, which is easily seen to be a club. Finally, apply Proposition~\ref{prop:subtle} to the club $C\cap\bar D$ and the sequence $\langle T'_\alpha\mid \alpha\in C\cap\bar D\rangle$, where $T'_\alpha=T_\alpha$ for regular $\alpha$ and $T'_\alpha=T_\alpha\cap \alpha$ otherwise, to obtain regular $\alpha<\beta$ such that $T_\alpha=T_\beta\cap\alpha$. But this is impossible because we assumed both that $T_\beta$ is ${<}\beta$-satisfiable and that $T_\alpha$ is not satisfiable.

Second, suppose that every logic $\cL$ has a stationary class of weak compactness cardinals.  Fix a class club $C$ of ordinals and a class sequence $ A = \seq{A_\xi\mid\xi\in{\rm Ord}}$ with $A_\xi\subseteq \xi$.  We can thin $C$ out to assume that $|V_\alpha|=\alpha$ for every $\alpha \in C$ and that every successor in $C$ has uncountable cofinality.  For each $\alpha \in C$, define the associated structure $$\bA_\alpha = \seq{V_\alpha, \in, D_\alpha, C \cap \alpha, A_\alpha},$$ where $D_\alpha = \{ (\gamma, \delta) \in \alpha\times\alpha \mid \delta \in A_\gamma\}$ codes all smaller $A_\gamma$.

Let $\tau$ be the language consisting of one binary relation and three unary relations, so that in particular $\bA_\alpha$ are $\tau$-structures. We define a logic $\bL^{C, A}$ which extends first-order logic by adding a sentence $\Psi$ such that a $\tau$-structure $\seq{M, E, J, K, L} \vDash \Psi$ if and only if there is some $\beta \in C$ such that $\seq{M, E, J, K, L}$ is isomorphic to $\bA_\beta$.  Let $\tau_\alpha$ be the language $\tau$ extended by adding constants $\{c_x\mid x\in V_\alpha\}\cup\{c\}$. For each $\alpha \in C$, define the $\bL^{C, A}(\tau_\alpha)$-theory
$$T_\alpha=\{\Psi\} \cup {\rm ED}(\bA_\alpha,c_x)_{x\in V_\alpha} \cup \{c \neq c_\beta \mid \beta < \alpha\},$$
where each constant $c_x$ is interpreted as $x$.  The theory $T_\alpha$ has size $\alpha$ and is ${<}\alpha$-satisfiable (in an expansion of $\bA_\alpha$).

By assumption, there is some limit point $\kappa^*$ of $C$ that is weakly compact for $\bL^{C, A}$.  So, in particular, $T_{\kappa^*}$ has some model $\cM$.  By definition of $\Psi$, $\cM \rest \tau$ is isomorphic to some $\bA_\delta$.  Since $\bA_\delta$ models ${\rm ED}(\bA_{\kappa^*},c_x)_{x\in V_{\kappa^*}}$, this induces an elementary embedding $\pi:\bA_{\kappa^*} \to \bA_\delta$.  The interpretation of $c$ witnesses that $\pi$ is not onto the ordinals of $\bA_\delta$, so let $\eta < \delta$ be the minimum ordinal not hit.  Note that $\eta \leq \kappa^*$.

If $\eta = \kappa^*$, then $\pi$ is the identity and $\bA_{\kappa^*} \prec \bA_\delta$.  Since the sequence $A$ is coded in the language, this means $A_\delta \cap \kappa^* = A_{\kappa^*}$.

If $\eta<\kappa^*$, then $\crit \pi=\eta$.  We claim that $\eta \in C$.  If not, then we can find $\nu, \mu \in C$ so $\nu$ is the largest member of $C$ below $\eta$ and $\mu$ is the smallest member of $C$ above it.  Since $\kappa^*$ is a limit member of $C$, we have
$$\nu < \eta < \mu < \kappa^*.$$
We know that $\mu$ is definable from $\nu$ (being its successor in $C$), so $\pi$ fixes $\mu$ since it fixes $\nu$.  It follows that $\pi \rest V_\mu$ is a nontrivial elementary embedding of $V_\mu$ into itself, which violates Kunen's Inconsistency, since by our assumption on successor elements of $C$, $\mu$ has uncountable cofinality. Thus, $\eta \in C$, which means $\pi(\eta) \in C$.  Applying the elementary embedding to the predicate $D_{\kappa^*}$, we have that $D_\eta = D_{\pi(\eta)} \cap\eta$, as desired.
\end{proof}

Now we revisit the caveats of this theorem.  First, we note that the statement about stationarily many weak compactness cardinals is not equivalent to the corresponding statement about existence.

\begin{theorem}\label{thm-not-stat}
It is consistent that every logic has a weak compactness cardinal, but not a stationary class of weak compactness cardinals.
\end{theorem}
\begin{proof}
The third author and Hamkins showed in \cite{gh-genvop} that there is a model of virtual Vopenka's principle in which ${\rm Ord}$ is not Mahlo.  Theorem~\ref{th:subtle} shows that if every logic has a stationary class of weak compactness cardinals, then ${\rm Ord}$ is subtle and hence, in particular, Mahlo (note that this direction of the theorem does not require global choice). Thus, in a model where virtual \Vopenka's principle holds, but ${\rm Ord}$ is not Mahlo, every logic has a weak compactness cardinal, but every logic cannot have a stationary class of weak compactness cardinals.
\end{proof}

Second, we show that the use of global choice in Theorem \ref{th:subtle} is necessary.  Global choice is natural as we think about models of the form $V_\kappa$, but raises interesting technical questions.  We begin with a lemma showing that the first weak compactness cardinal of $\bL^2$ is a regular limit cardinal.

\begin{lemma}\label{lem-wc-second}
If $\kappa$ is the least weak compactness cardinal for $\bL^2$, then $\kappa$ is a regular limit cardinal.
\end{lemma}

\begin{proof}
Let $\kappa$ be the least weak compactness cardinal for $\bL^2$. Every cardinal $\alpha<\kappa$ is not a weak compactness cardinal for $\bL^2$, so fix a ${<}\alpha$-satisfiable $\bL^2(\tau_\alpha)$-theory $T_\alpha=\langle \sigma^\alpha_\xi\mid \xi<\alpha\rangle$ of size $\alpha$ which does not have a model. We can assume that the language $\tau_\alpha$ is coded by elements of $\alpha$, and so are the sentences $\sigma^\alpha_\xi$. Fix some $M\prec H_{\kappa^+}$ of size $\kappa$ so  $\sigma(\alpha,\xi)=\sigma^\alpha_\xi$ is a function in $M$, and note that $\sigma(\alpha,\xi)\in \alpha$ by our coding assumption. Let $\tau$ be the language consisting of a binary relation $\in$ and constants $\{c_x\mid x \in M\}\cup\{c\}\cup\{c_M\}$. Let $T$ be the following $\bL^2(\tau)$-theory:
\begin{enumerate}
\item A large fragment of ${\rm ZFC}$,
\item Magidor's $\Phi$,
\item $\{S_\alpha\mid\alpha<\kappa\}$, where $S_\alpha:=$ ``For every $\beta<c_\alpha$, the theory $\{c_\sigma(c_\alpha,c_\xi)\mid \xi<\beta\}$ has a model."
\item $\{c_\xi\neq c<c_\kappa\mid \xi<\kappa\}$,
\item $c_M\vDash {\rm ED}(M,c_x)_{x\in M}$ is a transitive submodel.
\end{enumerate}
The theory $T$ clearly has size $\kappa$ and is ${<}\kappa$-satisfiable. Thus, some $V_\lambda$ models $T$ and there is an elementary embedding $j:M\to N$, where $N$ is the interpretation of $c_M$ in $V_\lambda$, sending $x\in M$ to the interpretation of $c_x$. Suppose the critical point of $j$ is some $\alpha<\kappa$. For every $\beta<j(\alpha)$, the theory $\{j(\sigma)(j(\alpha),\xi)\mid\xi<\beta\}$ has a model in $V_\lambda$. So in particular, the theory $\{j(\sigma)(j(\alpha),\xi)\mid\xi<\alpha\}$ has a model. By elementarity, $j(\sigma)(j(\alpha),\xi)=j(\sigma(\alpha,\xi))$ for every $\xi<\alpha$, and thus, the theory $T_\alpha$ has a model, which is impossible since $V_\lambda$ is correct about this. Thus, $\crit j=\kappa$. But once we have an elementary embedding $j:M\to N$ with $\crit j=\kappa$ for some $M\prec H_{\kappa^+}$ of size $\kappa$,  standard arguments show that $\kappa$ must be a regular limit cardinal.
\end{proof}

\begin{theorem}\label{th:globalChoiceCE}
It is consistent that there is a model of ${\rm ZFC}$ in which:
\begin{enumerate}
\item ${\rm Ord}$ is subtle,
\item there is no definable global well-ordering,
\item $\bL^2$ does not have a weak compactness cardinal.
\end{enumerate}
\end{theorem}


\begin{proof}
We work in $L$ and assume that the principle ${\rm Ord}$ is subtle holds there. Next, we force with the Easton-support class product $\p=\Pi_{\gamma\in\text{Reg}}\Add(\gamma,1)$, where $\Add(\gamma,1)$ is the forcing to add a Cohen subset to $\gamma$ with conditions of size less than $\gamma$. It will be convenient to assume that we skip stage $\omega$ and start with $\omega_1$. Let $L[G]$ be the resulting class forcing extension. Global choice fails in any forcing extension by $\p$ (the hypothesis $V=L$ plays no role here) for definable classes, and indeed, such a forcing extension cannot even have a definable global linear order \cite{hamkins:GlobalLinearOrder}. Fixing an ordinal $\delta$, let $$\p_{{\leq}\delta}=\Pi_{\gamma\in \text{Reg}\cap\delta+1}\Add(\gamma,1)$$ and let $$\p_{{>}\delta}=\Pi_{\gamma\in\text{Reg}\cap (\delta,{\rm Ord})}\Add(\gamma,1),$$ so that $\p=\p_{\leq\delta}\times \p_{{>}\delta}$. Correspondingly, let $G_{{\leq}\delta}=G\upharpoonright \p_{{\leq}\delta}$ and let $G_{{>}\delta}=G\upharpoonright \p_{{>}\delta}$, so that $G=G_{{\leq}\delta}\times G_{{>}\delta}$.

First, we argue that, in $L[G]$, we still have that ${\rm Ord}$ is subtle. To that end, fix a definable class club $C$ and a definable ${\rm Ord}$-length sequence $\vec A= \seq{A_\xi \mid \xi \in {\rm Ord}}$ with $A_\xi\subseteq \xi$ in $L[G]$.  We will reflect each of these to objects in $L$.  For $C$, suppose it is defined by the formula $\varphi(x,a)$ with parameter $a$. Let $\dot a$ be a $\p$-name such that $\dot a_G=a$. Clearly, there is an ordinal $\delta$ such that $\dot a$ is a $\p_{{\leq}\delta}$-name and $\dot a_{G_{{\leq}\delta}}=a$. We will argue that the club $C$ is definable already in $L[G_{{\leq}\delta}]$ by the formula $\psi(x,a):=\one_{\p_{{>}\delta}}\forces \varphi(x,\check a)$. The reason is that the forcing $\p_{{>}\delta}$ is weakly homogeneous (since it is a product of weakly homogeneous forcing notions) and therefore if a condition $p\forces\varphi(\check\alpha,\check a)$, then this must already be forced by $\one_{\p_{{>}\delta}}$. Next, let's observe that $C$ contains a club $\bar C$ that is definable in $L$. Let $\bar C=\{\alpha\in{\rm Ord}\mid \one_{\p_{{\leq}\delta}}\forces \psi(\check\alpha,\dot a)\}$. Clearly $\bar C$ is closed, so let's argue that it is unbounded. Choose some $\beta_0\gg\delta$. There must be a $\p_{{\leq}\delta}$-name $\dot c_0$ such that $\one_{\p_{{\leq}\delta}}\forces \psi(\dot c_0,\dot a)\wedge \dot c_0>\beta_0$. Since antichains of $\p_{{\leq}\delta}$ have size at most $|\p_{{\leq}\delta}|$, there is an ordinal $\beta_1$ such that $\one_{\p_{{\leq}\delta}}\forces \dot c_0<\beta_1$. Again, we can find a $\p_{{\leq}\delta}$-name $\dot c_1$ such that $\one_{\p_{{\leq}\delta}}\forces \psi(\dot c_1,\dot a)\wedge \dot c_1>\beta_1$, and then find $\beta_2$ such that $\one_{\p_{{\leq}\delta}}\forces \dot c_1<\beta_2$. Continuing in this manner, we can define a sequence $\seq{\dot c_n\mid n<\omega}$ of $\p_{{\leq}\delta}$-names and a sequence $\seq{\beta_n\mid n<\omega}$ of ordinals such that $\one_{\p_{{\leq}\delta}}\forces \psi(\dot c_n,\dot a)\wedge\check\beta_n<\dot c_n<\check\beta_{n+1}$. Let $\one_{\p_{{\leq}\delta}}\forces \dot c=\bigcup_{n<\omega}\dot c_n$ and let $\beta=\bigcup_{n<\omega}\beta_n$. Then $\one_{\p_{{\leq}\delta}}\forces \psi(\dot c,\dot a)\wedge \dot c=\check{\beta}$.

By similar arguments from the previous paragraph, we can assume that $\vec A$ is already definable in some $L[G_{{\leq}\delta}]$. For $\xi\in{\rm Ord}$, let $\dot A_\xi$ be a nice $\p_{{\leq}\delta}$-name for $A_\xi$. Observe that sufficiently high above $\delta$, we can assume, by coding appropriately, that $\dot A_\xi\subseteq \xi$. Thus, using that ${\rm Ord}$ is subtle in $L$, there must be $\alpha<\beta$ in $\bar C$ such that $\dot A_\beta\cap \alpha=\dot A_\alpha$. If the coding is chosen properly, we will then have that $A_\alpha\cap\beta=A_\beta$ as well. This finishes the argument that ${\rm Ord}$ is subtle holds in $L[G]$.

Second, we will argue that in $L[G]$, second-order logic $\bL^2$ cannot have a weak compactness cardinal. So suppose toward a contradiction that $\kappa$ is the least weak compactness cardinal for $\bL^2$ in $L[G]$. By Lemma \ref{lem-wc-second}, $\kappa$ is a regular limit cardinal, and hence inaccessible, since the ${\rm GCH}$ continues to hold in $L[G]$ by standard arguments.

For every regular $\gamma$, let $G_\gamma$ be the Cohen subset of $\gamma$ added by $G$ and observe that all initial segments of $G_\gamma$ are in $L$ by closure of $\Add(\gamma,1)$ and the fact that $\p$ is a product. Let $\tau$ be the language consisting of a binary relation $\in$, a unary predicate $D$, and constants $\{c_x\mid x\in L_\kappa\}\cup\{c\}$. Let $T$ be the $\bL^2(\tau)$-theory:
\begin{enumerate}
\item ${\rm ED}(L_\kappa,\in,G_\kappa,c_x)_{x\in L_\kappa}$,
\item $\{c_\xi<c<\kappa\mid \xi<\kappa\}$,
\item $D$ is a set of ordinals,
\item every initial segment of $D$ is in $L$,
\item $D$ is not in $L$.
\end{enumerate}
Statements (4) and (5) are second-order assertions, to verify which we have to ask whether there is an $L_\alpha$ witnessing them (which can be coded by a subset of the model). The theory $T$ is ${<}\kappa$-satisfiable, as witnessed by the structure $(L_\kappa,\in,G_\kappa)$, and has size $\kappa$. So by our assumption on $\kappa$, $T$ has a model $(L_\beta,\in,D)$ with $\beta>\kappa$. In particular, there is an elementary embedding $j:L_\kappa\to L_\beta$. If $j$ is non-trivial, it must have a critical point $\gamma<\kappa$, and the existence of such an embedding implies  $0^\sharp$ in $L[G]$. But this is impossible since $\p$ is countably closed, and hence cannot add a real. Thus, $j$ is the identity map, but then it follows that $G_\kappa\in L$, which is impossible. This final contradiction shows that there cannot be a weak compactness cardinal for $\bL^2$ in $L[G]$.
\end{proof}
Finally, it should be noted that Vop\v enka's principle is consistent with the failure of the existence of a global well-order.
\begin{theorem}
It is consistent that there is a model of ${\rm ZFC}$ in which\footnote{This result came out of a joint discussion with Andrew Brooke-Taylor and Asaf Karagila.}
\begin{enumerate}
\item  Vop\v enka's principle holds,
\item there is no definable global well-order.
\end{enumerate}
\begin{proof}
First, the argument at the start of the proof of Theorem \ref{th:globalChoiceCE} can modified to get the failure of global choice from an Easton-support  ${\rm Ord}$-length iteration adding a Cohen subset to every regular cardinal (rather than the product used there). The key to the argument is to show that the iteration is weakly homogeneous. Iterations of weakly homogeneous forcing notions can fail to be weakly homogeneous, but the result follows for iterations of weakly homogeneous forcing notions, where the forcings at each stage have canonical names \cite{DobrinenFriedman:homogeneousIterations}, which is definitely the case for $\Add(\gamma,1)$. Now, by a result of Brooke-Taylor \cite{Brooke-Taylor:Vopenka}, \Vopenka's principle is preserved by all progressively directed-closed\footnote{A forcing iteration is \emph{progressively directed-closed} if for every cardinal $\alpha$, the tail of the iteration is ${<}\alpha$-directed closed.} ${\rm Ord}$-length Easton-support iterations. Thus, starting in a universe satisfying Vop\v enka's principle, we can force to kill global choice, while preserving Vop\v enka's principle.
\end{proof}
\end{theorem}
\bibliographystyle{amsalpha}
\bibliography{largeCardinalLogics}

\end{document}